\documentclass[12pt]{amsart}
\usepackage{mathrsfs}
\usepackage{amssymb,amsmath}
\usepackage{setspace}
\usepackage{color}
\usepackage[all]{xy}
\usepackage{graphics}
\usepackage[dvips]{sseq}
\date{May 30, 2011 (revision)}
\calclayout
\newtheorem{dummy}{anything}[section]
\newtheorem{Theorem}[dummy]{Theorem}
\newtheorem*{Thm}{Theorem \ref{thm:typetwo}}
\newtheorem*{Thm1}{Theorem \ref{thm:one}}

\newtheorem{Lemma}[dummy]{Lemma}
\newtheorem{Proposition}[dummy]{Proposition}

\theoremstyle{definition}
\newtheorem{definition}[dummy]{Definition}
 
 \newtheorem*{construct}{Connected sum along a circle}
 
 \newtheorem{remark}[dummy]{Remark}
 
 \newtheorem*{acknowledgement}{Acknowledgement}

\newcommand{\z}{\mathbb Z}
\newcommand{\w}{\widetilde}
\newcommand{\rp}{\mathbb R \mathrm P}
\newcommand{\cp}{\mathbb C \mathrm P}
\DeclareMathOperator{\rk}{rank}
\newcommand{\scs}{\, \sharp_ {S^1}}
\newcommand{\Sharp}{\, \sharp\,}
\newcommand{\cxi}{c _1(\xi)}
\newcommand\abelian{fibered type\ }
\newcommand\abelianB{fibered type}
\DeclareMathOperator{\Pin}{Pin} 
\DeclareMathOperator{\TopPin}{TopPin}
\DeclareMathOperator{\Top}{Top}
\DeclareMathOperator{\Th}{Th}
\DeclareMathOperator{\Spin}{Spin}
\newcommand{\XX}[1]{X^5(#1)}
\newcommand{\vsp}{\phantom{$\Big ($}}

\begin{document}
\title{On Certain $5$-manifolds with Fundamental Group of Order $2$}
\author{Ian Hambleton}
\address{Department of Mathematics \& Statistics
 \newline\indent
McMaster University
 \newline\indent
Hamilton, ON  L8S 4K1, Canada}
\email{hambleton{@}mcmaster.ca}
\author{Yang Su}
\address{Hua Loo-Keng Key Laboratory of Mathematics
\newline \indent
Chinese Academy of Sciences
\newline \indent
Beijing, 100190, China}
\email{suyang{@}math.ac.cn}

\thanks{Research partially supported by
NSERC Research Grant A4000. The first author would like to thank the Max Planck Institut f\"ur Mathematik in Bonn for its hospitality and support.}

\begin{abstract}
In this paper, an explicit classification result for certain
$5$-manifolds with fundamental group $\z/2$ is obtained. These
manifolds include  total spaces of circle bundles over
simply-connected $4$-manifolds.
\end{abstract}

\maketitle

\section{Introduction}\label{sec:one}

The classification of manifolds with certain properties is a central
topic of topology, and in dimensions $\ge 5$ methods from handlebody theory and surgery
have been successfully applied  to a number of cases.
One of the first examples was the complete classification of
simply-connected $5$-manifolds  by Smale \cite{smale2} and Barden
\cite{barden1} in 1960's. This result has been very useful for studying the
existence of other geometric structures on $5$-manifolds, such as
the existence of Riemannian metrics with given curvature properties.
We consider this as a model and motivation for studying the classification of
non-simply connected $5$-manifolds.

\smallskip
An   orientable $5$-manifold $M$
 is said to be of \emph{\abelian}if $\pi_2(M)$ is
a trivial $\z [\pi_1(M)]$-module. In this paper, we will be
concerned with closed, orientable \abelian $5$-manifolds $M^5$ with
$\pi_1(M) \cong \z /2$, and torsion-free $H_2(M;\z)$. The
classification of these manifolds in the smooth (or PL) and 
topological categories is  given in Section \ref{sec:three}. We give a
simple set of invariants, namely the rank of $H_2(M ;\z)$ and the
$\Pin^{\dagger}$-bordism ($\TopPin^{\dagger}$-bordism) class 
 of a
characteristic submanifold,  which determine the diffeomorphism
(homeomorphism) types. Here is the main result in the smooth case.

\begin{Thm1}
Two smooth, closed, orientable \abelian $5$-manifolds $M$ and $M'$, 
with fundamental group $\z /2$ and torsion-free second homology
group, are diffeomorphic if and only if they have the same $w_2$-type,
$\rk H_ 2(M)=\rk H_ 2(M')$, and 
$[P]=[P'] \in \Omega_4^{\Pin^{\dagger}}/\pm$, 
where $P$ and $P'$ are characteristic
submanifolds and $\dagger = c, -, +$ for $w_2$-types \textup I,
\textup{II}, \textup{III} respectively.
\end{Thm1}

Here $\Omega_4^{\Pin^{\dagger}} /\pm$ denotes a quotient of the Pin-bordism group by a certain subgroup of order two (see Definition \ref{def:notation}). The Pin-bordism variants and the $w_2$-type notation are explained in Section \ref{sec:two}.

The homeomorphism classification is given in Theorem \ref{thm:two}.
We also determine all the relation among these invariants  (Theorem
\ref{thm:three}), and give a list of standard forms for these
manifolds (Theorem \ref{thm:four}, Theorem \ref{thm:four.five}).

\smallskip
One motivation for this classification problem comes from the study of
circle bundles $M^5$ over simply-connected $4$-manifolds, since
their total spaces are of \abelianB. Duan-Liang
\cite{duan-liang1} gave an explicit geometric description of $M^5$ for
simply-connected total spaces, making essential use of the results
of Smale and Barden. As an application of our results, in
Section \ref{sec:six}
we give an explicit geometric
description when the total spaces have fundamental group $\z /2$.
\begin{Thm}[type II]
Let $X$ be a closed, simply-connected, topological spin
$4$-manifold, $\xi\colon  S^1 \hookrightarrow M^5 \to X$ be a circle
bundle over $X$ with $c_1(\xi)=2\cdot$(primitive). Then we have
\begin{enumerate}\addtolength{\itemsep}{0.3\baselineskip}
\item if $KS(X)=0$, then $M$ is smoothable and $M$ is diffeomorphic
to $$(S^2 \times \mathbb R \mathrm P^3) \scs ((\sharp_k\,  S^2
\times S^2)\times S^1);$$

\item if $KS(X)=1$, then $M$ is non-smoothable and $M$ is
homeomorphic to $$*(S^2 \times \mathbb R \mathrm P^3) \scs
((\sharp_k\,  S^2 \times S^2)\times S^1).$$
\end{enumerate}
Where $k=\rk H_2(X)/2-1$.
\end{Thm}
In the statement, $*(S^2 \times \mathbb R \mathrm P^3)$ denotes a non-smoothable manifold homotopy equivalent to $S^2 \times \mathbb R \mathrm P^3$.
The corresponding results for the other $w_2$-types are given in
Theorem \ref{thm:typethree} and Theorem \ref{thm:typeone}.

Classification results can also be useful in studying the
 existence problem for
geometric structures on \abelian $5$-manifolds.  For example, a
closed, orientable $5$-manifold with $\pi_1=\z/2$, such that $w_2$ vanishes on homology, admits a contact structure  by
the work of Geiges and Thomas \cite{geiges-thomas1}.
 They showed that all such
manifolds can be obtained by surgery on $2$-dimensional links from
exactly one of ten model manifolds. 

The topology of such manifolds
of \abelian are described explicitly for the first time by our
results, and we note that all the manifolds listed in Theorem
\ref{thm:four} satisfy the necessary condition $W_3=0$ for the
existence of contact structures.
Our  results have already been used by Geiges and  Stipsicz \cite{geiges-stipsicz1} to prove new existence theorems for contact structures on $5$-manifolds.  It may be possible to obtain
similar information for \abelian $5$-manifolds which admit Sasakian
or Einstein metrics by using the work of Boyer and Galicki
\cite{ boyer-galicki1}.

The surgery exact sequence of Wall \cite{wall-book} provides a way
 to classify manifolds within a given (simple) homotopy type.
However, in the application to concrete problems, one often faces
homotopy theoretical difficulties. In our situation, the setting of
the problems is appropriate for the application of the modified
surgery methods developed by Kreck \cite{kreck3}. The proofs in
Section \ref{sec:four} and Section \ref{sec:five} are based on this
theory.

In dimension $5$, the smooth category and the PL category are equivalent.
By convention,  $M$ stands for either a smooth or a topological manifold
when not specified.

\begin{acknowledgement}
The second author would like to thank the
 James Stewart Centre for Mathematics at McMaster
  and the Institute for Mathematical
Sciences of the National University of Singapore for research visits
in Fall, 2008. The authors would like to thank M.~Olbermann for
helpful discussions.

\end{acknowledgement}

\section{Preliminaries}\label{sec:two}
\subsection{$\Pin^{\dagger}$-structures on vector bundles}
Recall that the groups $\Pin^{\pm}(n)$ are central extensions of $O(n)$ by $\z /2$
$$1 \to \mathbb Z/2 \to \Pin^{\pm}(n) \to O(n) \to 1,$$
and $\Pin^{c}(n)$ is a central extension of $O(n)$ by $U(1)$
$$1 \to U(1) \to \Pin^{c}(n) \to O(n) \to 1$$
(see \cite[\S 1]{kirby-taylor2} and \cite[\S 2]{hkt1}). Let $\dagger
\in \{c, +, - \}$. After stabilization we have classifying spaces
$B\Pin^{\dagger}$ and fibrations $B\Pin^{\dagger} \to BO $. A
$\Pin^{\dagger}$-structure on a stable vector bundle $\xi$ over a
space $X$ is a fiber homotopy class of lifts of a classifying map
$c_{\xi}\colon  X \to BO$ to $B\Pin^{\dagger}$.

\begin{Lemma}\cite[Lemma 1]{hkt1}
\begin{enumerate}
\item
A vector bundle $\xi$ over $X$ admits a $\Pin^{\dagger}$-structure if
and only if
$$\begin{array}{ll}
\beta (w _2(\xi)) =0 & \mathrm{for} \ \ \dagger =c, \\
w _2(\xi)=0 & \mathrm{for}\ \ \dagger =+, \\
w_ 2(\xi)=w_ 1(\xi)^2 & \mathrm{for}\ \ \dagger =-, \\
\end{array}$$
where $\beta \colon  H^2(X;\z/2) \to H^3(X;\z)$ is the Bockstein
operator induced from the exact coefficient sequence $\z \to \z \to
\z/2$.

\item
$\Pin^{\pm}$-structures are in bijection with $H^1(X;\z/2)$
and $\Pin^c$-structures are in bijection with $H^2(X;\z)$.
\end{enumerate}
\end{Lemma}

$\Pin^{\pm}$-structures on a vector bundle $\xi$ over $X$ are related to Spin-structures on an associated vector bundle:
\begin{Lemma}\cite[Lemma 1.7]{kirby-taylor2}\label{lem:pin-spin}
Let $\mathcal Spin(\xi)$ denote the set of equivalence classes of Spin structures on $\xi$, and $\mathcal Pin^{\pm}(\xi)$ denote the set of equivalence classes of $\Pin^{\pm}$-structures on $\xi$. There are bijections
$$\mathcal Pin^{-}(\xi) \to \mathcal Spin(\xi \oplus \det \xi)$$
$$\mathcal Pin^{+}(\xi) \to \mathcal Spin(\xi \oplus 3\det \xi).$$
\noindent
which are natural under the actions of $H^{1}(X;\z/2)$.
\end{Lemma}

It is well known that a $\Spin^c$-structure on a vector bundle $\xi$
is the same as a Spin-structure on $\xi \oplus \gamma$, where
$\gamma$ is a complex line bundle with $c_1(\gamma) \equiv w_2(\xi)
\pmod 2$ (see \cite[Cor.~D.4]{lawson-michelsohn1}). Similarly, a $\Pin^c$-structure on a
vector bundle $\xi$ may be viewed as a $\Pin^-$-structure on $\xi
\oplus \gamma$, where $\gamma$ is a complex line bundle with
$c_1(\gamma) \equiv w_1(\xi)^2 +w_2(\xi) \pmod 2$.


\subsection{$w_{2}$-types and characteristic submanifolds}\label{two-type}
Let $M$ be a closed, orientable $5$-manifold with $\pi_1(M) \cong
\z/2$ and universal cover $\w{M}$. The manifold $M$ is said to be
of $w_{2}$-\emph{type \textup{I}} if $w_ 2(\w{M}) \ne
0$,  of $w_{2}$-\emph{type \textup{II}} if $w_2(M)=0$, and
of $w_{2}$-\emph{type \textup{III}} if $w_2(M) \ne 0$ and $w_
2(\widetilde{M})=0$.
. 
By the universal
coefficient theorem, there is an exact sequence
$$0 \to \mathrm{Ext}(H_ 1(M), \mathbb Z/2) \to H^2(M;\mathbb Z/2) \to \mathrm{Hom}(H_ 2(M), \mathbb Z/2) \to 0.$$

\begin{Lemma}
$M$ is of type 
type \textup{III} $\Leftrightarrow w_ 2(M) \ne 0$ and $w_ 2(M) \in \mathrm{Ext}(H_ 1(M), \mathbb
Z/2)$.
\end{Lemma}

\begin{proof}
There is a commutative diagram
$$\xymatrix{0 \ar[r] &  \mathrm{Ext}(H_ 1(M), \mathbb Z/2) \ar[r] \ar[d] &  H^2(M;\mathbb Z/2) \ar[r] \ar[d]& \mathrm{Hom}(H_ 2(M), \mathbb Z/2) \ar[r] \ar[d]& 0 \\
            0 \ar[r] &  \mathrm{Ext}(H_ 1(\w{M}), \mathbb Z/2) \ar[r]  & H^2(\w{M};\mathbb Z/2) \ar[r]  & \mathrm{Hom}(H_ 2(\w{M}), \mathbb Z/2)
            \ar[r]
& 0.     }$$ Let $p\colon \w{M} \to M$ be the covering map, then
$T\w{M}=p^*TM$ and $w _2(\w{M}) =p^*w_ 2(M)$. By the exact sequence
$\pi_ 2(M) \to H_ 2(M) \to H_ 2(\z/2) \to 0$ and the fact $H_
2(\z/2)=0$ (cf. \cite{brown1}), it is seen that the map $H_ 2(\w{M})
\to H_ 2(M)$ is surjective, therefore the last vertical map in the
diagram $\mathrm{Hom}(H _2(M), \mathbb Z/2) \to \mathrm{Hom}(H_
2(\w{M}), \mathbb Z/2)$ is a monomorphism. Thus $w_ 2(\w{M})=0$ if
and only if $w_ 2(M) \in \mathrm{Ext}(H_ 1(M), \mathbb Z/2)$.
\end{proof}

\begin{remark} By this Lemma, the type II and type III
manifolds are manifolds having second Stiefel-Whitney class equal to
zero on homology. The existence of contact structures on these manifolds is shown in \cite{geiges-thomas1}.
\end{remark}

Recall that for a manifold $M^n$ with fundamental group $\z /2$, a
\emph{characteristic submanifold} $P^{n-1} \subset M$ is defined as follows (see \cite{lopez-de-medrano1} and \cite[\S 5]{geiges-thomas1}): there is a
decomposition $\w{M}=A \,\cup\, TA$ such that $\partial A =\partial (TA)=
\w{P}$, where $T$ is the deck-transformation. Then $P:= \w{P}/T$ is
called the characteristic submanifold of $M$. For example, if
$M=\rp^n$, then $P=\rp^{n-1}$. In general, let $f\colon M \to \rp^N$ ($N$
large) be the classifying map of the universal cover, transverse to
$\rp^{N-1}$, then $P$ can be taken as $f^{-1}(\rp^{n-1})$. By
equivariant surgery we may assume that $\pi_1(P) \cong \z/2$ and
that the inclusion $i\colon  P \subset M$ induces an isomorphism on
$\pi_1$. Different characteristic submanifolds of $M$ are bordant, where a bordism
is obtained from a homotopy between the relevant classifying maps. The above construction also holds in the topological
category by  topological transversality \cite{kirby-siebenmann1}.

In the smooth category, the division of the manifolds under
consideration into three $w_{2}$-types corresponds to different
$\Pin^{\dagger}$-structures on their characteristic submanifolds,
compare \cite[Lemma 9]{geiges-thomas1} for $\dagger = \pm$.

\begin{Lemma}\label{lem:TP}
Let $M$ be a smooth, orientable $5$-manifold with $\pi_1(M) \cong
\z/2$ and $H_ 2(M;\z)$ torsion-free. Let $P \subset M$ be a
characteristic submanifold (with $\pi_ 1(P)\cong \pi_{1}(M)$). Then
$TP$ admits a $\Pin^{\dagger}$-structure, where
$$\dagger = \left \{ \begin{array}{ll}
c & \textrm{if}\ \ M \ \ \textrm{is of type \textup I} \\
- & \textrm{if}\ \ M \ \ \textrm{is of type \textup{II}} \\
+ & \textrm{if}\ \ M \ \ \textrm{is of type \textup{III}} \\
\end{array} \right.$$
More precisely, if $M$ is of type \textup{II}, then a Spin-structure on $TM$
gives a $\Pin^-$-structure on $TP$; if $M$ is of type \textup{III}, then a
Spin-structure on $TM \oplus 2L$ gives a $\Pin^+$-structure on $TP$,
where $L$ is the nontrivial line bundle over $M$; if $M$ is of type
\textup{I}, then a Spin-structure on $TM \oplus \gamma$ gives a
$\Pin^c$-structure on $TP$, where $\gamma$ is a complex line bundle
over $M$ such that $c_1(\gamma) \equiv w_2(M) \pmod 2$.
\end{Lemma}

\begin{proof}
Let $i\colon P \subset M$ be the inclusion and $\nu$ be the normal
bundle of this inclusion, then $TP \oplus \nu = i^*TM$.
If $M$ is of type II,  a Spin-structure on $TM$ induces a
Spin-structure on $TP\oplus \nu = TP \oplus \det TP$, therefore by
Lemma \ref{lem:pin-spin}, gives a $\Pin^-$-structure on $TP$.

If $M$ is of type III, then $TM \oplus 2L$ admits Spin-structures
and such a structure induces a Spin-structure on $TP \oplus 3 \det
TP$, henceforth a $\Pin^+$-structure on $TP$.

If $M$ is of type I, then $TP$ has neither $\Pin^-$ nor
$\Pin^+$-structures. Now $TM \oplus \gamma$ has Spin-structures.
Such a structure induces a Spin-structure on $TP \oplus \det TP
\oplus i^*\gamma$, and hence a $\Pin^-$-structure on $TP \oplus
i^*\gamma$. Since $c_1(i^*\gamma) \equiv i^*w_2(M) = w_1(P)^2 +
w_2(P) \pmod 2$, we obtain a $\Pin^c$-structure on $TP$.
\end{proof}

\begin{Lemma}\label{lem:Pbordant}
If $M$ is of type \textup{II} or \textup{III}, then different characteristic
submanifolds of $M$ with the $\Pin^{\pm}$-structures obtained by
\textup{Lemma \ref{lem:TP}} represent a pair of mutually inverse elements in
the corresponding bordism group $\Omega_4^{\Pin^{\pm}}$.
\end{Lemma}
\begin{proof}
If we fix a Spin-structure on $TM$ (or $TM \oplus 2L$), then it's
clear that all different characteristic submanifolds with the
induced $\Pin^{\pm}$-structure are $\Pin^{\pm}$-bordant, for they
are transversal preimages of classifying maps of $\pi_1(M)$ and all
such maps are homotopic. Now we fix a characteristic submanifold
$P$, then the two $\Pin^{\pm}$-structures on $TP$ are related by the
action of $w_1(P)$, and it's a general fact that $P$ with such two
$\Pin^{\pm}$-structures give rise to a pair of mutually inverse
elements in the corresponding bordism group
\cite[p.190]{kirby-taylor2}.
\end{proof}

\section{Main Results}\label{sec:three}
Now we are ready to state the classification of the manifolds under consideration.

\begin{Theorem}\label{thm:one}
Two smooth, closed, orientable \abelian $5$-manifolds $M$ and $M'$,
with fundamental group $\z /2$ and torsion-free second homology
group, are diffeomorphic if and only if they have the same $w_2$-type,
$\rk H_ 2(M)=\rk H_ 2(M')$, and $[P]=[P'] \in \Omega_
4^{\Pin^{\dagger}} /\pm$, where $P$ and $P'$ are the characteristic
submanifolds and $\dagger = c, -, +$ for types \textup{I},
\textup{II}, \textup{III} respectively.
\end{Theorem}

\begin{remark} It is known that
$\Omega_ 4^{\Pin^-}=0$ \cite{kirby-taylor2}. Therefore, 
$\rk H_ 2(M)$ is the only diffeomorphism invariant for the type II manifolds.
\end{remark}

There are topological versions of the central extensions mentioned
above and we have groups $\TopPin^{\dagger}(n)$, $\dagger \in \{c, +,
- \}$. For the preliminaries
on $\TopPin^{\dagger}(n)$ we refer to \cite{kirby-taylor2} and
\cite{hkt1}. Therefore we have corresponding results in the
topological category.

\begin{Lemma}
Let $M$ be a topological, orientable $5$-manifold with $\pi_ 1(M)
\cong \z/2$ and $H _2(M;\z)$ torsion-free. Let $P \subset M$ be a
characteristic submanifold (with $\pi_ 1(P)\cong \pi_{1}(M)$). Then $TP$
admits a $\TopPin^{\dagger}$-structure, where
$$\dagger = \left \{ \begin{array}{ll}
c & \textrm{if}\ \ M \ \ \textrm{is of type \textup I} \\
- & \textrm{if}\ \ M \ \ \textrm{is of type \textup{II}} \\
+ & \textrm{if}\ \ M \ \ \textrm{is of type \textup{III}} \\
\end{array} \right.$$
\end{Lemma}

\begin{Theorem}\label{thm:two}
Two topological, closed, orientable \abelian $5$-manifolds $M$ and
$M'$, with fundamental group $\z /2$ and torsion-free second homology
group, are homeomorphic if and only if they have the same $w_2$-type,
$\rk H _2(M)=\rk H_ 2(M')$ and $[P]=[P'] \in \Omega
_4^{\TopPin^{\dagger}} /\pm$, where $P$ and $P'$ are
characteristic submanifolds and $\dagger = c, -, +$ for type \textup
I, \textup{II}, \textup{III} respectively.
\end{Theorem}

\smallskip
The groups $\Omega _4^{\Pin^{\pm}}$ and $\Omega_ 4^{\TopPin^{\pm}}$
are computed in \cite{kirby-taylor2}. $\Omega_ 4^{\TopPin^c}$ is
computed in \cite[p.654]{hkt1}. (Note that the r\^ole of $\Pin^+$ and
$\Pin^-$ in \cite{hkt1} are reversed since in that paper the authors
consider normal structures whereas here we use the convention in
\cite{kirby-taylor2}, looking at the tangential Gauss-map.) In a
similar way we will compute $\Omega_ 4^{\Pin^c}$ below. We list the
values of these groups:

\smallskip

\begin{center}
\begin{tabular}{|c|c|c|c|}
\hline
\vsp
       $\dagger$ & $\Omega_ 4^{\Pin^{\dagger}}$ & invariants     &  generators \\[0.5ex]
\hline
\vsp
 $c$      & $\z /8 \oplus \z /2 $      & (arf, $w_ 2^2$) & $\rp ^4$, $\cp ^2$ \\[.5ex]
       $+$       & $\z /16$                   & ?           & $\rp ^4$ \\[.5ex]
       $-$       & $0$                        &  --           & -- \\[.5ex]
\hline
\vsp
       $\dagger$ & $\Omega_ 4^{\TopPin^{\dagger}}$ & invariants & generators \\[0.5ex]
\hline
\vsp
 $c$      & $\z /2 \oplus \z /8 \oplus \z /2$ & ($KS$, arf, $w_ 2^2$) & $E_8$, $\rp ^4$, $\cp ^2$ \\[.5ex]
       $+$       & $\z /2 \oplus \z /8$             & ($KS$, arf) & $E_8$, $\rp ^4$ \\[.5ex]
       $-$       & $\z /2$ &   $KS$                    &  $E_8$\\[.5ex]
\hline
\end{tabular}
\end{center}

\medskip

\noindent Computation of $\Omega_ 4^{\Pin^c}$: the extension
$$1\to \Pin^- \to \Pin^c \to U(1) \to 1$$
induces Gysin-sequence (compare \cite[p.654]{hkt1})
$$\cdots \to \Omega _4^{\Pin^-} \to \Omega_ 4^{\Pin^c} \stackrel{\cap\,
c}{\longrightarrow} \Omega_ 2^{\Pin^-}(BU(1)) \to \Omega_ 3^{\Pin^-}
\to \cdots .$$ Since $\Omega_ 4^{\Pin^-}=\Omega_ 3^{\Pin^-}=0$ (see
\cite{kirby-taylor2}), we have an isomorphism $$\Omega_ 4^{\Pin^c}
\stackrel{\cap\, c}{\longrightarrow} \Omega_ 2^{\Pin^-}(BU(1))$$ and
the latter group is the same as $\Omega_ 2^{\TopPin^-}(BU(1))$,
which is computed in \cite{hkt1}.
The invariants in Theorem \ref{thm:one} are subject to certain
relations.
\begin{definition}\label{def:notation}
 Denote $r=\rk H_ 2(M)$, $q=[P] \in \Omega_ 4^{\Pin^+}
/\pm =\{0,1, \dots , 8 \}$ and $(q,s)=[P] \in \Omega_ 4^{\Pin^c} /
\pm =\{0,1, \dots , 4 \} \times \{0,1 \}$. 
\end{definition}

As an application of the
semi-characteristic class \cite{lee1}, we have
\begin{Theorem}\label{thm:three}
Let $M$ be a smooth, orientable $5$-manifold with $\pi_ 1(M) \cong \z
/2$ and torsion-free $H_ 2(M)$, having the invariants as above. Then
these invariants subject to the following relations
\medskip

\begin{center}
\begin{tabular}{|c|c|}
\hline
type  &  relation \\[.3ex]
\hline \vsp
 \textup{I} & $q+s+r \equiv 1$ $\pmod{2}$ \\[.3ex]
       \textup{II} & $r \equiv 1$ $\pmod{2}$ \\[.3ex]
       \textup{III} & $q+r \equiv 1$ $\pmod{2}$\\[.3ex]
       \hline
\end{tabular}
\end{center}
\end{Theorem}

\smallskip

Now we give a list of all the manifolds under consideration,
realizing the possible invariants. We need some preliminaries.

By a computation of the surgery exact sequence, it is shown in
\cite{wall-book} that in the smooth (or PL) category, there are $4$
distinct diffeomorphism types of manifolds which are homotopy
equivalent to $\rp ^5$, these are called fake $\rp ^5$. An explicit
construction using links of singularities (Brieskorn spheres) can be
found in \cite{geiges-thomas1}. Following the notations there, we
denote these fake $\rp ^5$ by $\XX{q}$, $q=1,3,5,7$, with $\XX{1}
=\rp^5$. These manifolds fall into the class of manifolds under
consideration. They are of type III and the $\Pin^+$-bordism class
of the corresponding characteristic submanifold is $q \in \Omega_
4^{\Pin^+} /\pm = \{0,1, \dots , 8 \}$, see \cite{geiges-thomas1}. In
our list of standard forms these fake projective spaces will serve
as building blocks under the operation $\scs$---``connected-sum
along $S^1$'', which we explain now, compare \cite{hkt1}.

\begin{construct}
Let $M _i$ $(i=1,2)$ be oriented $5$-manifolds with fundamental group
$\z /2$ or $\z$, and at least one of the fundamental groups is $\z
/2$. Denote the trivial oriented $4$-dimensional real disc bundle over $S^1$
by $E$. Choose embeddings of $E$ into $M_ 1$ and $M_2$,
representing a
 generator of $\pi _1 (M_ i)$, such that the first embedding preserves
  the orientation and the second reverses it. Then we define
$$M _1\scs M _2 := (M_ 1-E)\cup _{\partial}\, (M_ 2-E).$$
Note that if one of the $5$-manifolds admits an
orientation reversing automorphism, then the construction doesn't
depend on the orientations, and this is the case for the building
blocks in the list below, namely, $S^2 \times \rp ^3$, $S^2 \times
S^2 \times S^1$, $\XX{q}$ and $\cp ^2 \times S^1$ admit orientation
reversing automorphisms.  (The fact that $\XX{q}$ admits orientation
reversing automorphisms follows from that $\rp ^5$ admits
orientation reversing automorphisms and that the action of
$Aut(\rp^5)$ on the structure set $\mathscr{S}(\rp^5)$ is trivial.)

The Seifert-van Kampen theorem implies that $\pi_ 1(M_ 1 \scs M_ 2) \cong \z
/2$. The Mayer-Vietoris exact sequence implies that $H_ 2(M_ 1 \scs
M_ 2)$ is torsion-free, and hence $M_ 1 \scs M_ 2$ is of fibered type.
The homology  $\rk H _2(M_ 1 \scs M_ 2) = \rk H_ 2(M_ 1) +
\rk H_ 2(M_ 2) +1$ if both fundamental groups are $\z/2$, and $\rk H_
2(M_ 1\scs M_ 2) =\rk H_ 2(M_ 1)  + \rk H_ 2(M_ 2)$ if one
of the fundamental groups is $\z$.

Since $\pi_ 1 SO(4) \cong \mathbb Z/2$, there are actually two
possibilities to form $M_ 1 \scs M_ 2$. However, from the
classification result, it turns out that this ambiguity happens only
when we construct $\XX{q}\scs\XX{q'}$. This does depend on
the framings, and therefore $\XX{q}\scs \XX{q'}$ represents two
manifolds. Note that the characteristic submanifold of
$M_ 1\scs M_ 2$ is $P_ 1\scs P_ 2$ (see \cite[p.651]{hkt1}
for the definition of $\scs$ for nonorientable $4$-manifolds with
fundamental group $\z /2$). Therefore if we fix $\Pin^+$-structures
on each of the characteristic submanifolds, then
$\XX{q}\scs \XX{q'}$ is well-defined.
\end{construct}

This construction allows us to construct manifolds with a given
bordism class of characteristic submanifold. Note that $P _1\scs P_
2$ corresponds to the addition in the bordism group $\Omega_
4^{\Pin^{\dagger}}$. Now for $q=0,2,4,6,8$, choose $l,l' \in
\{1,3,5,7\}$ and appropriate $\Pin^+$-structures on the
characteristic submanifolds of $\XX{l}$ and $\XX{l'}$, we can form a
manifold $\XX{l} \scs \XX {l'}$ such that the characteristic
submanifold $[P]=q \in \Omega_ 4^{\Pin^+} /\pm$. We denote this
manifold also by $\XX{q}$. For example, we can form $\XX{0}=\XX{1}
\scs \XX{1}$ and $\XX{2}=\XX{1} \scs \XX{1}$ with different glueing
maps.

With these notations, the list of standard forms of the manifolds
under consideration is given as follows:

\begin{Theorem}\label{thm:four}
Every closed smooth orientable \abelian $5$-manifold with fundamental
group $\z /2$ and second homology group $\z ^r$ is diffeomorphic to
exactly one of the following standard forms:

\

type \textup{I :} $\XX{q} \scs (S^2 \times \rp ^3) \scs ((\sharp_k\,
S^2 \times S^2)\times S^1)$, $r=2k+(5+(-1)^q)/2$, $q \in \{0, \dots,
4\}$;

\

\hspace{1.2cm} $\XX{q} \scs (\cp ^2 \times S^1) \scs ((\sharp_k\,
S^2 \times S^2)\times S^1)$, $r=2k+(3+(-1)^q)/2$, $q \in \{0, \dots,
4\}$;

\

type \textup{II :} $(S^2 \times \mathbb R \mathrm P^3) \scs
((\sharp_k\, S^2 \times S^2)\times S^1)$, $r=2k+1$;

\

type \textup{III :} $\XX{q}\scs ((\sharp_k\,  S^2 \times S^2)\times
S^1)$,
 $r=2k+(1+(-1)^q)/2$, $q \in
\{0, \dots, 8\}$.

\

\noindent Where $\sharp_k\, S^2 \times S^2$ is the connected sum of $k$
copies of $S^2 \times S^2$.
\end{Theorem}

\begin{remark} There can be other descriptions of the
manifolds in the list. For example, we have a (more symmetric)
description of the type II standard forms
$$\underbrace{(S^2 \times
\mathbb R \mathrm P^3) \scs  \cdots  \scs  (S^2 \times
\mathbb R \mathrm P^3)} _{k \ \ \mathrm{times}}.$$
\end{remark}

\begin{remark}
Note that the universal covers of the manifolds under consideration
have torsion-free second homology, therefore, according to the
results of Smale and Barden, are diffeomorphic to $\sharp_r(S^2
\times S^3)$ or $B\sharp_{r-1}(S^2 \times S^3)$, where $B$ is the
nontrivial $S^3$-bundle over $S^2$. From this point of view, Theorem
\ref{thm:four} gives the classification of orientation preserving
free involutions on $\sharp_r(S^2 \times S^3)$ and
$B\sharp_{r-1}(S^2 \times S^3)$, which act trivially on $H_2$. For
example, consider the orientation preserving free involution on $S^2
\times S^3$ given by $(x,y) \mapsto (r(x), -y)$, where $r \colon S^2
\to S^2$ is the reflection along a line and $- \colon S^3 \to S^3$
is the antipodal map. Then the quotient space is actually the sphere
bundle of the nontrivial orientable $\mathbb R^3$-bundle over
$\rp^3$. From Theorem \ref{thm:one} it is easy to see that this is
just $X^5(0)$.
\end{remark}

\begin{remark} The above list may be of use in the study of
geometric structures on these manifolds. Geiges and Thomas
\cite{geiges-thomas1} show
that the type II and type III manifolds admit contact structures. On
the other hand, a necessary condition for the existence of contact
structures on $M^{2n+1}$ is the reduction of the structure group of
$TM$ to $U(n)$, hence the vanishing of integral Stiefel-Whitney
classes $W_{2i+1}(M)$. It is easy to see that the type I manifolds
satisfy this necessary condition. These manifolds also satisfy the
necessary conditions on the cup length and Betti numbers in
 \cite{ boyer-galicki1} for the existence of Sasakian structures.
Therefore it would be interesting to study these geometric
structures on these manifolds.
\end{remark}

\begin{proof}[The proof of Theorem  \ref{thm:four}]
By the Van-Kampen theorem and the Mayer-Vietoris sequence it is easy
to see that all the manifolds in the list are orientable, with
fundamental group $\z/2$ and torsion-free $H_ 2$, and the
$\pi_ 1$-action on $H_ 2$ is trivial. Therefore we only need to verify that these
manifolds have different invariants and realize all the possible
invariants.

\smallskip
\noindent
\underline{Type II:} $\rk H_ 2((S^2 \times \mathbb R \mathrm P^3)
\scs ((\sharp_k\,  S^2 \times S^2)\times S^1))=2k+1$.

\smallskip
\noindent
 \underline{Type III:} the characteristic submanifold of
$\XX{q}\scs ((\sharp_k\,  S^2 \times S^2)\times S^1)$ is just that
of $\XX{q}$, which corresponds to $q \in \Omega_ 4^{\Pin^+}/\pm
=\{0, \cdots, 8 \}$.

\smallskip
\noindent
\underline{Type I:}  similarly, the manifold  $\XX{q} \scs (\cp ^2 \times S^1) \scs ((\sharp_k\,
S^2 \times S^2)\times S^1)$ has characteristic
submanifold invariant $(q,1) \in \Omega_ 4^{\Pin^c}/\pm$.
\end{proof}

To give a list of standard forms of the manifolds under
consideration in the topological case, we need a topological
$5$-manifold which is homotopy equivalent to $S^2 \times \rp^3$ and
whose characteristic submanifold represents the nontrivial element
in $\Omega_4^{\TopPin^-}=\z/2$. Note that by Theorem \ref{thm:two},
if such manifolds exist, then the homeomorphism type is unique.
Following the notation in \cite{hkt1}, we denote this manifold by
$*(S^2 \times \rp^3)$. We now give the construction of $*(S^2 \times
\rp^3)$.

Let $W=S^2 \times \rp^3 \scs E_8 \times S^1$, so that
$\pi_1(W) =\z/2$ and the characteristic submanifold of $W$ is $S^2
\times \rp^2 \Sharp E_8$. Let $h \colon W \to S^2 \times \rp^3$ be a
degree $1$ normal map which extends the degree $1$ normal map $f
\colon S^2 \times \rp^2 \Sharp E_8 \to S^2 \times \rp^2$. Then by
doing codimension $1$ surgery on $h$ we obtain a $W'$ with
characteristic submanifold $P=*(S^2 \times \rp^2)$ and a degree $1$
normal map $h' \colon W' \to S^2 \times \rp^3$ extending a homotopy
equivalence $f' \colon
*(S^2 \times \rp^2) \to S^2 \times \rp^2$ (cf. \cite{hkt1} for the construction of $*(S^2 \times \rp^2)$). The $\pi$-$\pi$
theorem allows us to do further surgeries on the complement of a
tubular neighbourhood of $P$ to obtain a homotopy equivalence.

In the topological category there are four fake $\rp^5$'s. Two of them
are smoothable. We denote these manifolds  by $X^5(p,q)$ ($p\in \{
0,1\}$, $q \in \{1,3\}$) such that the characteristic submanifold of
$X^5(p,q)$ is $(p,q) \in \Omega_4^{\TopPin^+}/\pm = \{ 0,1\} \times
\{0,1,2,3,4\}$. Similar to the smooth case, we can also construct
$X^5(p,q)$ ($p\in \{ 0,1\}$, $q \in \{0,2,4\}$) by circle connected
sum of fake $\rp^5$. (Note that the Kirby-Siebenmann invariant is
additive under the connected sum operation \cite{quinn2}).

\begin{Theorem}\label{thm:four.five}
Every closed topological orientable \abelian $5$-manifold with
fundamental group $\z /2$ and second homology group $\z ^r$ is
homeomorphic to exactly one of the following standard forms:

\

type \textup{I :} $X^5(p,q) \scs (S^2 \times \rp ^3) \scs ((\sharp
k\, S^2 \times S^2)\times S^1)$,

\smallskip

\hspace{2cm} $r=2k+(5+(-1)^q)/2$, $q \in \{0, \dots, 4\}$, $p=0,1$;

\

\hspace{1.2cm} $X^5(p,q) \scs (\cp ^2 \times S^1) \scs ((\sharp k\,
S^2 \times S^2)\times S^1)$,

\smallskip

\hspace{2cm} $r=2k+(3+(-1)^q)/2$, $q \in \{0, \dots, 4\}$, $p=0,1$;

\

type \textup{II :} $(S^2 \times \mathbb R \mathrm P^3) \scs ((\sharp
k\, S^2 \times S^2)\times S^1)$, $r=2k+1$;

\

\hspace{1.3cm} $*(S^2 \times \rp^3) \scs ((\sharp k\, S^2 \times
S^2)\times S^1)$, $r=2k+1$;

\

type \textup{III :} $X^5(p,q) \scs ((\sharp k\,  S^2 \times
S^2)\times S^1)$,
 $r=2k+(1+(-1)^q)/2$, $q \in
\{0, \dots, 4\}$, $p=0,1$.

\end{Theorem}

\smallskip

From the above list, we can also give a homotopy classification.
\begin{Theorem}\label{thm:five}
The homotopy type of $M^5$ is determined by its $w_2$-type, $\rk H _2(M)$,
and in the type \textup{I} case the number $\langle w_ 2(M)^2\,\cup\, t
+ t^5, [M] \rangle \in \z/2$, where $t \in H^1(M;\z/2)$ is the
nonzero element.
\end{Theorem}
\begin{proof}
Note that $\XX{q}$  and $X^5(p,q)$ are homotopy equivalent to
$\rp^5$ and the operation $\scs$ preserves homotopy equivalence.
This proves the theorem for the type II and III cases. For type I
manifolds, the $s$-component of the characteristic submanifold $P$
is determined by $\langle w_ 2(P)^2, [P]\rangle$. Since $w
_2(P)=i^*(w_ 2(M)+t^2)$, $\langle w _2(P)^2, [P]\rangle =\langle w_
2(M)^2\cup t + t^5, [M] \rangle$, and this is a homotopy invariant.
\end{proof}

\section{Bordism and Surgery}\label{sec:four}

\subsection{The framework of modified surgery}
The main tool used in our solution of the classification problem is
the modified surgery developed by Kreck \cite{kreck4}, \cite{kreck3}. We first briefly describe how this theory is applied in our situation.

Let $p\colon  B \to BO$ be a fibration, and
$\bar{\nu}\colon  M^{2m-1} \to B$ be a lift of the normal Gauss map
$\nu\colon  M \to BO$ classifying the stable normal bundle of $M$. Such a lift $\bar{\nu}$ is called a \emph{normal $B$-structure}
of $M$, and the pair $(M, \bar{\nu})$ is called a \emph{normal $k$-smoothing in $B$} if the map $\bar{\nu}$ is a
$(k+1)$-equivalence. Manifolds with normal $B$-structures form a bordism
theory $\Omega_*(B,p)$, described in Stong \cite[Chap.~II]{stong1}.

Suppose $(M^{2m-1}_ i, \bar{\nu_ i})$ ($i=1,2$) are two normal
$(m-1)$-smoothings in $B$, and suppose that $(W^{2m}, \bar{\nu})$ is a $B$-bordism between
$(M^{2m-1}_ 1, \bar{\nu_ 1})$ and $(M^{2m-1}_ 2, \bar{\nu _2})$. Then the surgery obstruction for $W^{2m}$ being $B$-bordant rel.~boundary to an $s$-cobordism (implying that $M_ 1$ and
$M_ 2$ are diffeomorphic) is a $(-1)^{m}$-quadratic form over $(\Lambda, S)$, where $\Lambda =\z[\pi_{1}(B)]$ is the group ring and $S \subset \Lambda$ is a certain form-parameter subgroup. The surgery obstruction lies in an abelian group $L_{2m}^{s, \tau}(\pi_{1}(B), w_{1}(B), S)$ (\cite[Theorem 5.2 b]{kreck4}), where $w_{1}(B)$ is the orientation character. This group is related to Wall's $L$-group in the following diagram (\cite[p.37]{kreck4})
$$\xymatrix@R-5pt{
0 \ar[r] & L^{s}_{2m}(\pi_{1}, w_{1}) \ar[r] & L^{s, \tau}_{2m}(\pi_{1}, w_{1}) \ar[r] \ar[d] & Wh(\pi_{1}) \\
              &                                                              & L^{s, \tau}_{2m}(\pi_{1}, w_{1}, S) \ar[d] & \\
              &                                                              & 0 &
              }$$
where $Wh(\pi_{1})$ is the Whitehead group (see \cite{milnor1}).

In our case, $\pi_ 1=\z/2$, $Wh(\z/2)=0$ and $L^{s}_{6}(\z/2)=\z/2$.
Therefore our surgery obstruction group is either $0$ or $\z/2$. In
the latter case, it is isomorphic to $L_{6}^{s}(\z/2)$, the
non-trivial element  is detected by the Kervaire-Arf invariant (see Wall \cite[\S 13A]{wall-book}).
Since the closed manifold  $S^3\times S^3$ admits a framing with Arf invariant 1, we may eliminate the surgery obstruction by connected sum in the interior of $W$.
We have the following:

\begin{Proposition}\label{prop:one}
Two smooth $5$-manifolds $M_1$ and $M_2$ with fundamental group
$\z/2$ are diffeomorphic if they have bordant normal $2$-smoothings
in some fibration $B$.
\end{Proposition}

The fibration $B$ is called the \emph{normal $2$-type of $M$} if $p$ is
$3$-coconnected. This is an invariant of $M$. Because of this
proposition, the solution to the classification problem consists of
two steps: first, determine the normal $2$-types $B$ for the
$5$-manifolds under consideration, and then determine invariants to detect the  corresponding bordism groups $\Omega _5(B,p)$.

\subsection{Normal $2$-types}
Let $M^5$ be a \abelian $5$-maniofold. The universal coefficient
theorem implies that $H_ 2(\widetilde{M}) \otimes _{\z[\pi_ 1]} \z
\to H_ 2(M)$ is an isomorphism. Since the $\pi_ 1(M)$-action on $H_
2(\widetilde{M})$ is trivial, we have $H_ 2(\widetilde{M}) \otimes
_{\z[\pi_ 1]} \z =H_ 2(\widetilde{M})$, therefore $H_
2(\widetilde{M}) \to H_ 2(M)$ is an isomorphism, also is the second
Hurewicz map $\pi_ 2(M) \to H_ 2(M)$. Now suppose $\pi_ 1(M) \cong
\z/2$ and $H_ 2(M) \cong \z^r$.

We start with the description of the normal $2$-types for type II
manifolds. It is the simplest situation and illuminates the ideas.

\smallskip
\noindent
\underline{Type II:} consider the fibration
$$p\colon  B=\rp^{\infty} \times
(\cp^{\infty})^r \times B\Spin \to BO,$$ where $p\colon  B \to BO$
is trivial on the first two factors and on $B\Spin$ it is the
canonical projection from $B\Spin$ onto $BO$. A lift $\bar{\nu}  \colon M \to
B$ is given as follows: the map to $\rp^{\infty}$ is the classifying
map of the fundamental group; choose a basis $\{u _1, \dots, u_ r\}$
of the free part of $H^2(M) \cong \z^r \oplus \z/2$, by realizing
each element $u_ i$ by a map to $\cp^{\infty}$ we get a map to
$(\cp^{\infty})^r$; a $\Spin$-structure on $\nu M$ gives rise to a
map to $B\Spin$. It's easy to see that $(B,p)$ is the normal
$2$-type of type II manifolds and that $\bar{\nu}$ induces an
isomorphism on $\pi_ 1$ and $H_2$. Since the second Hurewicz maps
$\pi_ 2(M) \to H_ 2(M)$ and $\pi_ 2((\cp^{\infty})^r) \to H_
2((\cp^{\infty})^r)$ are isomorphisms, $\bar{\nu}$ is a normal
$2$-smoothing.

\smallskip
\noindent
 \underline{Type III:} let $\eta$ be the canonical real
line bundle over $\rp^{\infty}$, and $2\eta=\eta \oplus \eta$. Consider
the fibration
$$p\colon  B=\rp^{\infty} \times
(\cp^{\infty})^r \times B\Spin \stackrel{f _1 \times f_
2}{\longrightarrow} BO \times BO \stackrel{\oplus}{\longrightarrow}
BO,$$ where $f_ 1\colon  \rp^{\infty} \times (\cp^{\infty})^r \to
BO$ is the classifying map of $p_{1}^*(2\eta)$, (where $p_{1} \colon
\rp^{\infty} \times (\cp^{\infty})^r \to \rp^{\infty}$ is the
projection map,) $f_ 2\colon  B\Spin \to BO$ is the canonical
projection and $\oplus\colon  BO \times BO \to BO$ is the $H$-space
structure on $BO$ induced by the Whitney sum of vector bundles. A
lift $\bar{\nu} \colon M \to B$ is given as follows: the map to $\rp^{\infty}
\times (\cp^{\infty})^r$ is the same as in type II. Since $w_
2(2\eta)=w_ 1(\eta)^2$ is the nonzero element in $\mathrm{Ext}(H_
1(\rp^{\infty}), \z/2)$ and $w_ 2(M)$ is the nonzero element in
$\mathrm{Ext}(H_ 1(M), \z/2)$, we have $w_ 2(\bar{\nu}^*2\eta) =w_
2(\nu M)$. This implies that $\nu M - \bar{\nu}^*2\eta$ admits a
$\Spin$-structure. Such a structure induces a map to $B\Spin$. Then
$\bar{\nu}$ is a lift of $\nu$. It is easy to see that $(B,p)$ is
the normal $2$-type of type III manifolds and $\bar{\nu}$ is a
normal $2$-smoothing.

\smallskip
\noindent
\underline{Type I:} let $\gamma$ be the canonical complex line
bundle over $\cp^{\infty}$. Consider the fibration
$$p\colon  B=\rp^{\infty} \times
(\cp^{\infty})^r \times B\Spin \stackrel{f_ 1 \times f_
2}{\longrightarrow} BO \times BO \stackrel{\oplus}{\longrightarrow}
BO,$$ where $f _1\colon  \rp^{\infty} \times (\cp^{\infty})^r \to
BO$ is the classifying map of $p_{2}^*\gamma$, $p_{2}\colon
\rp^{\infty} \times (\cp^{\infty})^r \to \cp^{\infty}$ is the
projection map to the first $\cp^{\infty}$. A lift $\bar{\nu} \colon M \to B$ is
given as follows: since the Bockstein homomorphism $\beta \colon
H^2(M;\z/2) \to H^3(M;\z)$ is trivial, $w_ 2(M)$ is the mod $2$
reduction of an integral cohomology class. Since $w _2(M)$ is not
contained in $\mathrm{Ext}(H_ 1(M),\z/2)$, this integral cohomology
class can be taken as a primitive one, say, $u_1$ and we extend it
to a basis $\{u_ 1, \dots, u_ r\}$. Then the map to $\rp^{\infty}
\times (\cp^{\infty})^r$ is the same as above. Now $\nu M -
\bar{\nu}^* \gamma$ admits a $\Spin$-structure, this gives rise to a
map $M \to B\Spin$. Then $\bar{\nu}$ is a lift of $\nu$. It is easy
to see that $(B,p)$ is the normal $2$-type of type I manifolds and
$\bar{\nu}$ is a normal $2$-smoothing.

\subsection{Computation of the bordism groups}
In this subsection we calculate the bordism groups $\Omega_{5}{(B,p)}$ for our types:
$$\Omega_5^{\Spin}(\rp^{\infty} \times (\cp^{\infty})^r), \
\Omega_5^{\Spin}(\rp^{\infty} \times (\cp^{\infty})^r; p_{1}^*2\eta), \
\Omega_5^{\Spin}(\rp^{\infty} \times (\cp^{\infty})^r;
p_{2}^*\gamma).$$ The main tools are the Atiyah-Hirzebruch spectral sequence and the Adams spectral sequence.
Before doing the calculation, we need to compute  $\Omega_5^{\Spin}(\rp^{\infty})$, $\Omega_5^{\Spin}(\rp^{\infty}; 2\eta)$ and $\Omega_{5}^{\Spin}(\rp^{\infty} \times \cp^{\infty}; p_{2}^{*}\gamma)$. These groups can be calculated via the Adams spectral sequence. Here we give an alternative argument, emphasizing the role of the characteristic submanifolds.

There are long exact sequences (this is a special case of
\cite[(3.2)]{galatius-madsen-tillmann-weiss1})
$$\dots \to\Omega_n^{\Spin} \to \Omega_n^{\Spin}(\rp^{\infty};k\eta)
\stackrel{\partial}{\rightarrow}
\Omega_{n-1}^{\Spin}(\rp^{\infty};(k+1)\eta) \to
\Omega_{n-1}^{\Spin}\to \dots$$ and
$$\dots \to \Omega_n^{\Spin}(\cp^{\infty};\gamma) \to \Omega_n^{\Spin}(\rp^{\infty} \times \cp^{\infty};p_{2}^{*}\gamma)
\stackrel{\partial}{\rightarrow} \Omega_{n-1}^{\Spin}(\rp^{\infty}
\times \cp^{\infty};\eta\times \gamma) \to \dots$$
 where the maps
$\partial$ correspond to taking a characteristic submanifold.
In particular we have an isomorphism
$\Omega_5^{\Spin}(\rp^{\infty}) \stackrel{\cong}{\rightarrow}
\Omega_4^{\Spin}(\rp^{\infty};\eta)$, together with exact sequences
$$0 \to \Omega_5^{\Spin}(\rp^{\infty}; 2\eta) \to
\Omega_4^{\Spin}(\rp^{\infty};3\eta)$$
and
$$
 \Omega_{5}^{\Spin}(\cp^{\infty};\gamma) \to\Omega_{5}^{\Spin}(\rp^{\infty} \times \cp^{\infty}; p_{2}^{*}\gamma)
\to \Omega_{4}^{\Spin}(\rp^{\infty} \times \cp^{\infty}; \eta \times \gamma)\to\Omega_{4}^{\Spin}(\cp^{\infty};\gamma)
$$
Furthermore, we have
$$\Omega_n^{\Spin}(\rp^{\infty};\eta) \cong \Omega_n^{\Pin^-}, \
\Omega_n^{\Spin}(\rp^{\infty};3\eta) \cong \Omega_n^{\Pin^+}, \
\Omega_n^{\Spin}(\rp^{\infty} \times \cp^{\infty};\eta \times
\gamma) \cong \Omega_n^{\Pin^c}.$$ 
This is seen as follows: first,
given $[X^{n}, f] \in \Omega_{n}^{\Spin}(\rp^{\infty}; \eta)$,
clearly $$w_{1}(f^{*}\eta)=w_{1}(X)=w_{1}(\det TX).$$
Therefore by
Lemma \ref{lem:pin-spin},  the Spin-structure on $TX \oplus
f^{*}\eta$ induces a $\Pin^{-}$-structure on $TX$ and we have a
well-defined map $\Omega_n^{\Spin}(\rp^{\infty};\eta) \to
\Omega_n^{\Pin^-}$. Given $X^{n}$ together with a
$\Pin^{-}$-structure, by letting $f \colon X \to \rp^{\infty}$ be
the classifying map for $w_{1}(X)$, we obtain $[X, f] \in
\Omega_n^{\Spin}(\rp^{\infty};\eta)$. These two maps are inverse to
each other.  The $\Pin^{+}$ and $\Pin^{c}$ cases are similar.

The $\Pin^{\pm}$-bordism groups in low dimensions were calculated in
\cite{kirby-taylor2}: we have $\Omega_{4}^{\Pin^{-}}=0$ and
$\Omega_{4}^{\Pin^{+}} \cong \z/16$, generated by $\pm \rp^{4}$.
Also it is clear that under the map $$\Omega_5^{\Spin}(\rp^{\infty};
2\eta) \to \Omega_4^{\Spin}(\rp^{\infty};3\eta)\cong
\Omega_{4}^{\Pin^{+}},$$
the element $[\rp^{5}, \mathrm{inclusion}]$ goes to $\pm
\rp^{4}$, therefore the map $$\Omega_5^{\Spin}(\rp^{\infty}; 2\eta)
\to \Omega_4^{\Spin}(\rp^{\infty};3\eta)$$ is an isomorphism. An easy
Atiyah-Hirzebruch spectral sequence calculation shows that
$\Omega_{5}^{\Spin}(\cp^{\infty};\gamma) \cong \w
\Omega_{7}^{\Spin}(\cp^{\infty})=0$ and
$\Omega_{4}^{\Spin}(\cp^{\infty};\gamma) \cong \w
\Omega_{6}^{\Spin}(\cp^{\infty}) \cong \z \oplus \z$. Therefore the
map $\Omega_{5}^{\Spin}(\rp^{\infty} \times \cp^{\infty};
p_{2}^{*}\gamma) \to \Omega_{4}^{\Spin}(\rp^{\infty} \times
\cp^{\infty}; \eta \times \gamma)$ is also an isomorphism. To
summarize, we have

\begin{Lemma}
Taking characteristic submanifolds gives isomorphisms
$$\Omega_{5}^{\Spin}(\rp^{\infty}) \cong  \Omega_{4}^{\Pin^{-}}, \ \Omega_{5}^{\Spin}(\rp^{\infty}; 2\eta)  \cong \Omega_{4}^{\Pin^{+}}, \
\Omega_{5}^{\Spin}(\rp^{\infty} \times \cp^{\infty}; p_{2}^{*}\gamma) \cong \Omega_{4}^{\Pin^{c}}.$$
\end{Lemma}

Now we begin the calculation of the bordism groups of interest. As in the last subsection, we start with the type II manifolds,
which is the simplest case.

\medskip
\noindent
\underline{Type II:} recall that the normal $2$-type is
$$p\colon  B=\rp^{\infty} \times
(\cp^{\infty})^r \times B\Spin \to BO,$$ where $p\colon  B \to BO$
is trivial on the first two factors and is the canonical projection
from $B\Spin$ onto $BO$. Therefore the bordism group $\Omega_
5{(B,p)}$ is the $\Spin$-bordism group $\Omega_
5^{\Spin}(\rp^{\infty} \times (\cp^{\infty})^r)$. To compute this
bordism group, we apply the Atiyah-Hirzebruch spectral sequence. The
$E^2$-terms are $E^2 _{p,q}=H _p(\rp^{\infty} \times
(\cp^{\infty})^r;\Omega_ q^{\Spin})$.

To illuminate the situation, we first consider  the group $\Omega_
5^{\Spin}(\rp^{\infty} \times \cp^{\infty})$. The relevant terms and
differentials in the spectral sequence are depicted as follows:

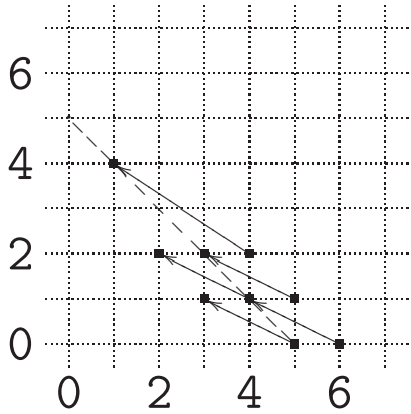
\begin{center}

\def\sseqgridstyle{\ssgridgo}
\sseqentrysize=.6cm
\def\sseqpacking{\sspackhorizontal}
\begin{sseq}{8}{8}
\ssmoveto 1 4  \ssdropbull \ssdashedline {2}{-2}\ssdropbull
\ssdashedline {1}{-1}\ssdropbull \ssdashedline {1}{-1} \ssdropbull
\ssarrow {-2}{1} \ssdropbull

\ssmoveto 4 1 \ssarrow {-2}{1}  \ssdropbull

\ssmoveto 5 1   \ssdropbull \ssarrow {-2}{1}

\ssmoveto 4 2  \ssdropbull \ssarrow {-3}{2}

\ssmoveto 1 4 \ssdashedline {-1}{1} \ssdrop{\cdot}

\ssmoveto 6 0 \ssdropbull \ssarrow{-2}{1}

\ssmoveto 7 0
\end{sseq}
\end{center}

\noindent
The $E^2$-terms are:
\begin{itemize}\addtolength{\itemsep}{0.2\baselineskip}
\item $E^2_{1,4}=H_1(\rp^{\infty} \times \cp^{\infty})\cong \z/2$,
\item $E^2_{2,2}=H_2(\rp^{\infty} \times \cp^{\infty};\z/2) \cong \z/2 \oplus \z/2$,
\item
$E^2_{3,1}=E^2_{3,2}=H_3(\rp^{\infty} \times \cp^{\infty};\z/2)
\cong \z/2 \oplus \z/2$,
\item $E^2_{4,1}=E^2_{4,2}=H_4(\rp^{\infty}
\times \cp^{\infty};\z/2) \cong (\z/2)^3$,
\item
$E^2_{5,0}=H_5(\rp^{\infty} \times \cp^{\infty}) \cong (\z/2)^3$,
\item
$E^2_{5,1}=H_5(\rp^{\infty} \times \cp^{\infty};\z/2) \cong
(\z/2)^3$,
\item  $E^2_{6,0}=H_6(\rp^{\infty} \times \cp^{\infty}) \cong
\z/2$.
\end{itemize}

\medskip
\noindent
The differential $d_ 2 \colon E^{2}_{p,1} \to E^{2}_{p-2,2}$ is dual to the Steenrod square
$$Sq^2 \colon H^{p-2}(\rp^{\infty} \times \cp^{\infty};\z/2) \to H^{p}(\rp^{\infty} \times \cp^{\infty};\z/2);$$
the differential $d_ 2 \colon E^{2}_{p,0} \to E^{2}_{p-2,1}$ is the mod $2$ reduction composed with the dual of the Steenrod square
$$H_{p}(\rp^{\infty} \times \cp^{\infty};\z) \to H_{p}(\rp^{\infty} \times \cp^{\infty};\z/2) \stackrel{(Sq^{2})^*}{\longrightarrow} H_{p-2}(\rp^{\infty} \times \cp^{\infty};\z/2).$$
With these identifications, the differentials $d_{2}$ starting from or ending at the line $p+q=5$ are easily computed.  Let $\alpha
\in H^1(\rp^{\infty};\z/2)$, $\beta \in H^2(\cp^{\infty};\z/2)$ denote
the generators, then on the $E^{3}$-page, we have three nontrivial terms in the line $p+q=5$: $E_{5,0}^{3} = \z/2$, dual to $\alpha^{3} \beta$; $E^{3}_{4,1} = \z/2$, dual to $\alpha^{2} \beta$; and $E^{3}_{1,4} =\z/2$. The terms $E^{3}_{5,0}$ and $E_{4,1}^{3}$ must survive to infinity, for there are no non-trivial differentials starting from or ending at these two positions (see the picture above). 

There is a possibly non-trivial differential $d_{3} \colon E^{3}_{4,2} \to E^{3}_{1,4}$.
To see this differential is indeed non-trivial, we just need to note that the terms $E^{3}_{1,4}=E^{2}_{1,4}=H_{1}(\rp^{\infty} \times \cp^{\infty};\z/2)$ come from $\rp^{\infty}$ and $\Omega_{5}^{\Spin}(\rp^{\infty})=\Omega_{4}^{\Pin^{-}}=0$. Therefore on the
$E^{\infty}$-page, in the line $p+q=5$, the nontrivial terms are
$$E^{\infty}_{5,0}=H_3(\rp^{\infty})\otimes H_2(\cp^{\infty}) \cong
\z/2, \ \ \ E^{\infty}_{4,1}=H_2(\rp^{\infty};\z/2)\otimes
H_2(\cp^{\infty};\z/2) \cong \z/2.$$

The calculation is finished once the extension problem is solved. We state the result in the following lemma.  Let $\tau \colon \cp^{\infty} \to \cp^{\infty}$ be
the involution on $\cp^{\infty}$ with $\tau_*=-1$ on
$H_2(\cp^{\infty})$, then $\tau$ induces an involution $\tau_{*}$ on $\Omega_ 5^{\Spin}(\rp^{\infty} \times
\cp^{\infty})$. Let $\alpha
\in H^1(\rp^{\infty};\z/2)$, $\beta \in H^2(\cp^{\infty};\z/2)$ be the nonzero elements.

\begin{Lemma}\label{lem:nonsplit}
The short exact sequence
$$0 \to \z/2 \to \Omega_ 5^{\Spin}(\rp^{\infty} \times
\cp^{\infty}) \to \z/2\to 0$$ is nonsplit, thus $\Omega_
5^{\Spin}(\rp^{\infty} \times \cp^{\infty}) \cong \z/4$. The elements $\pm 1 $
are represented by $\rp^{3} \times \cp^{1} \hookrightarrow
\rp^{\infty} \times \cp^{\infty}$. A bordism class $[X^5,f]$ equals
$\pm 1$ if and only if
 $\langle  \alpha ^3 \,\cup\, \beta, f_ *[X]
\rangle =1 \in \z/2$. There is a relation $\langle \alpha \,\cup\,
\beta^2 , f_
*[X] \rangle =0$. The action $\tau_{*}$ is the multiplication by $-1$.
\end{Lemma}
\begin{proof}
There is a product map
$$ \varphi \colon \Omega_3^{\Spin}(\rp^{\infty}) \otimes
\Omega_2^{\Spin}(\cp^{\infty}) \to \Omega_ 5^{\Spin}(\rp^{\infty}
\times \cp^{\infty}),$$ induced by the product of manifolds. There
is a corresponding product map on the Atiyah-Hirzebruch spectral
sequences
$$\Phi \colon E^r_{p,q}(1) \otimes E^r_{s,t}(2) \to E^r_{p+s, q+t}(3),$$
where on the $E^{\infty}$-page $\Phi$ is compatible with the
filtrations on the bordism groups and on the $E^2$-page it is just
the cross product map (see \cite[p.~352]{switzer1}). It is easy to see that
 the
Atiyah-Hirzebruch spectral sequence of
$\Omega_2^{\Spin}(\cp^{\infty})$ collapses on the line $p+q=2$. Also since $\Omega_3^{\Spin}(\rp^{\infty}) \cong \Omega_2^{\Pin^-} \cong \z/8$ (by \cite{kirby-taylor2}), we see that  the Atiyah-Hirzebruch spectral sequence of
$\Omega_3^{\Spin}(\rp^{\infty})$ collapses on the line $p+q=3$. From the knowledge of $E^{\infty}_{5,0}(3)$ and  $E^{\infty}_{4,1}(3)$ discussed above, we see there are surjections
$$\Phi \colon E^{\infty}_{3,0}(1) \otimes E^{\infty}_{2,0}(2) \to E^{\infty}_{5,0}(3), \  \Phi \colon E^{\infty}_{2,1}(1) \otimes E^{\infty}_{2,0}(2) \to E^{\infty}_{4,1}(3).$$
Therefore $\varphi$ is surjective. Now
$\Omega_3^{\Spin}(\rp^{\infty}) \cong  \z/8$ is
generated by $[\rp^3, \mathrm{inclusion}]$ and $\Omega_2^{\Spin}(\cp^{\infty})
\cong \Omega_2^{\Spin} \oplus H_2(\cp^{\infty})$.
The group $\Omega_2^{\Spin}$
is generated by $T^2$ with the Lie group spin structure. The product
$\varphi(\rp^3, T^2)=0$, since the map $\rp^3 \times T^2 \to
\rp^{\infty} \times \cp^{\infty}$ factors through $\rp^{\infty}$ and
$\Omega_5^{\Spin}(\rp^{\infty})=0$. Therefore we have a surjection
$$\z/8 \otimes \z \to \Omega_ 5^{\Spin}(\rp^{\infty} \times
\cp^{\infty}).$$ This shows that $\Omega_ 5^{\Spin}(\rp^{\infty}
\times \cp^{\infty}) \cong \z/4$, generated by $\rp^{3} \times
\cp^{1} \hookrightarrow \rp^{\infty} \times \cp^{\infty}$ and $[X,
(\mathrm{id}_{\rp^{\infty}} \times \tau) \circ f]=-[X,f]$. The fact
that a bordism class $[X^5,f]$ equals $\pm 1$ if and only if
 $\langle  \alpha ^3 \,\cup\, \beta, f_ *[X]
\rangle =1 \in \z/2$ comes from the fact that $E^{\infty}_{5,0}$ is dual to $\alpha^{3} \beta$. The
relation $\langle \alpha \,\cup\,
\beta^2 , f_
*[X] \rangle =0$ comes from the fact that the dual of $d_2$ maps $\alpha
\beta$ to $\alpha \beta^2$.
\end{proof}

In general, on the $E^{\infty}$-page of the Atiyah-Hirzebruch spectral sequence for $\Omega_
5^{\Spin}(\rp^{\infty} \times (\cp^{\infty})^r)$, the nontrivial terms
in the line $p+q=5$ are
\begin{eqnarray*}
E^{\infty}_{5,0} & = & \bigoplus_i H_3(\rp^{\infty})\otimes
H_2(\cp^{\infty}_i) \oplus \bigoplus _{i \ne j} H_1(\rp^{\infty})
\otimes H_2(\cp^{\infty}_i) \otimes H_2(\cp^{\infty}_j) \\
& \cong  & (\z/2)^{r+r(r-1)/2}\\
E^{\infty}_{4,1} & = & \bigoplus_i
H_2(\rp^{\infty};\z/2)\otimes H_2(\cp^{\infty}_i;\z/2)  \\
   & \cong &  (\z/2)^r
\end{eqnarray*}

Using the same argument as in Lemma \ref{lem:nonsplit}, we have the
following:
\begin{Proposition}[type II]\label{prop:two}
The bordism group $\Omega_ 5^{\Spin}(\rp^{\infty} \times
(\cp^{\infty})^r)$ is isomorphic to $(\z/4)^r \oplus
(\z/2)^{r(r-1)/2}$. Let $\alpha \in H^1(\rp^{\infty};\z/2)$,
$\beta_i \in H^2(\cp^{\infty}_ i;\z/2)$ be the nonzero elements, $\tau_{i}$ be the involution on $\cp^{\infty}_{i}$ with $\tau_{i*}=-1$ on $H_{2}$,
then
\begin{enumerate}
\item the $\z/2$-factors are determined by the invariants $\langle
\alpha \,\cup\, \beta_i \,\cup\, \beta_j, f_
*[X] \rangle \in \z/2$, with $i, j=1, \cdots r$, and $i>j$,
\item a bordism class $[X,f]$  has component $\pm 1$ in the $i$-th $\z/4$-factor if and only if $\langle \alpha ^3 \,\cup\, \beta _i , f_
*[X] \rangle =1 \in \z/2$, $i=1, \cdots r$,
\item there are relations $\langle \alpha \,\cup\, \beta_i^2 , f_
*[X] \rangle =0$ for all $i$,
\item the action of $\tau_{i}$ on the bordism group is multiplication by $-1$ on the $i$-th $\z/4$-factor and trivial on other factors.
\end{enumerate}
\end{Proposition}

\medskip
\noindent
\underline{Type III:} the normal $2$-type is
$$p\colon  B=\rp^{\infty} \times
(\cp^{\infty})^r \times B\Spin \to BO,$$ where the map on
$\rp^{\infty}$ is the classifying map of the vector bundle $2\eta$.
Therefore the bordism group $\Omega_ 5{(B,p)}$ is the twisted
$\Spin$-bordism group
$$\Omega_ 5^{\Spin}(\rp^{\infty}
\times (\cp^{\infty})^r;p_{1}^*2\eta) = \widetilde{\Omega}
_7^{\Spin}(\Th(p_{1}^*2\eta)).$$

In the Atiyah-Hirzebruch spectral sequence, the $E^2$-terms are
$$E^2_ {p,q}=\w H_ p(\Th(p_{1}^*2\eta);\Omega_q^{\Spin}).$$
 Since $2\eta$
is orientable, we may apply the Thom isomorphism and after a degree
shift $p \mapsto p-2$ we have  $E^2_{p,q}=H_p(\rp^{\infty} \times
(\cp^{\infty})^r;\Omega_q^{\Spin})$. Therefore the $E^2$-terms are
the same as in the type II case, and in the identification of the differentials $d _2$, we need to replace $Sq^{2}$ by $Sq^2+w_ 2(2\eta)$.

As before, we first look at the group $\Omega_{5}^{\Spin}(\rp^{\infty} \times \cp^{\infty};p_{1}^{*}2\eta)$. Clearly $\Omega_{5}^{\Spin}(\rp^{\infty};2\eta) \cong \z/16$ is a direct summand. Besides this, there are two terms on the $E^{\infty}$-page at positions $(5,0)$ and $(4,1)$ respectively, each isomorphic to $\z/2$.  The extension problem is solved in the following lemma.

\begin{Lemma}
We have an isomorphism
$$\Omega _5^{\Spin}(\rp^{\infty}
\times \cp^{\infty};p_{1}^*(2\eta)) \cong \z/4 \oplus \Omega _4^{\Pin^+}.$$
A bordism class $[X,f]$ has component $\pm 1$ in the $\z/4$-factor if and only if $\langle \alpha ^3 \,\cup\, \beta  , f_
*[X] \rangle =1\in \z/2$.  There is a relation $\langle \alpha^{3} \,\cup\, \beta, f_{*}[X]\rangle = \langle \alpha \,\cup\, \beta^{2}, f_{*}[X]\rangle$. The action $\tau_{*}$ of  the involution $\tau$ on $\cp^{\infty}$ is the multiplication by $-1$ on the $\z/4$-factor and trivial on the $\Omega _4^{\Pin^+}$-factor.
\end{Lemma}
\begin{proof}
From the above discussion we have
$$\Omega _5^{\Spin}(\rp^{\infty}
\times \cp^{\infty};p_{1}^*(2\eta)) \cong G \oplus \Omega _4^{\Pin^+},$$
where the order of $G$ is $4$.
To determine $G$, the geometric argument in Lemma \ref{lem:nonsplit} doesn't work since now we have $\Omega_{3}^{\Spin}(\rp^{\infty};2\eta) \cong \Omega_{2}^{\Pin^{+}} =0$. Thus we turn to consider the Adams spectral sequence for $\w \Omega_{t-s-2}^{\Spin}(\rp^{\infty}
\times \cp^{\infty};p_{1}^*(2\eta))=\pi^{S}_{t-s}(Th(p_{1}^*(2\eta))\wedge MSpin))$ at prime $2$:
$$Ext^{s,t}_{\mathcal A}(H^{*}(Th(p_{1}^*(2\eta))\wedge MSpin; \mathbb F_{2}); \mathbb F_{2}) \Rightarrow \pi^{S}_{t-s}(Th(p_{1}^*(2\eta))\wedge MSpin))/\textrm{non $2$-torsion},$$
where $\mathcal A$ is the mod $2$ Steenrod algebra.
We have
$$H^{*}(Th(p_{1}^*(2\eta))\wedge MSpin;\mathbb F_{2}) \cong \w H^{*}( Th(p_{1}^*(2\eta);\mathbb F_{2}) \otimes_{\mathbb F_{2}} H^{*}(MSpin;\mathbb F_{2}),$$
where $\w H^{*}( Th(p_{1}^*(2\eta);\mathbb F_{2})$ is a free $\mathbb F_{2}[t,x]$-module on one generator $u_{2}$ of degree $2$ (the Thom class), where $\deg t=1$ and $\deg x=2$, and
$$Sq(u_{2})=u_{2} + t^{2}u_{2}.$$
From this we may write down the $\mathcal A$-module structure of   $H^{*}(Th(p_{1}^*(2\eta))\wedge MSpin;\mathbb F_{2})$ in degree$\le 9$, produce a minimal free $\mathcal A$-resolution of $H^{*}(Th(p_{1}^*(2\eta))\wedge MSpin;\mathbb F_{2})$ which corresponds to the $E_{2}$-term of the spectral sequence. In practice, we may ignore the pure terms from $\rp^{\infty}$, since we already know the contribution of $\rp^{\infty}$ is a $\z/16$-summand.

In low degrees, the $E_{2}$-page of the spectral sequence is
depicted as follows (with horizontal index $t-s$ and vertical index
$s$. The calculation is confirmed by Olbermann and Abczynski using a
computer program developed by Bruner):
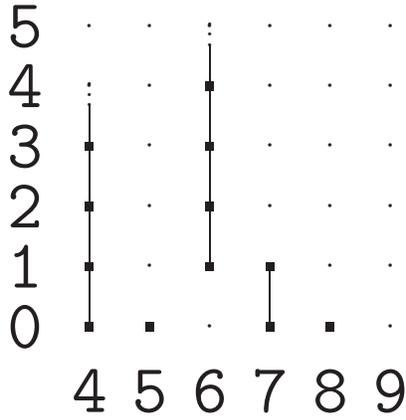
\begin{center}

\sseqxstep=1 \sseqystep=1

\def\sseqgridstyle{\ssgriddots}
\sseqxstart=4
\sseqentrysize=.8cm
\def\sseqpacking{\sspackhorizontal}
\begin{sseq}{6}{6}
\ssmoveto 4 0  \ssdropbull \ssline {0}{1}\ssdropbull
\ssline {0}{1}\ssdropbull \ssline {0}{1} \ssdropbull
\ssline {0}{1} \ssdrop{\vdots}

\ssmoveto 5 0  \ssdropbull

\ssmoveto 6 1  \ssdropbull \ssline {0}{1}\ssdropbull
\ssline {0}{1}\ssdropbull \ssline {0}{1} \ssdropbull
\ssline {0}{1} \ssdrop{\vdots}

\ssmoveto 7 0  \ssdropbull \ssline {0}{1} \ssdropbull

\ssmoveto 8 0 \ssdropbull
\end{sseq}
\end{center}
This shows that $G \cong \z/4$.

The fact that the generators of the $\z/4$-factor are detected by  the
invariant $\langle \alpha ^3 \,\cup\, \beta  , f_
*[X] \rangle \in \z/2$ and the relation $\langle \alpha^{3} \,\cup\, \beta, f_{*}[X]\rangle = \langle \alpha \,\cup\, \beta^{2}, f_{*}[X]\rangle$ are seen from the Atiyah-Hirzebruch spectral sequence, as in the type II case. From this, we claim that $[X^{5}(0),f]$ represents a generator of $\z/4$, where
$$f \colon X^{5}(0) \to \rp^{\infty} \times \cp^{\infty}$$ is a normal $2$-smoothing. To see this, recall that
$$X^{5}(0)=(\rp^{5}-S^{1}\times D^{4}) \,\cup_{\partial}\, (\rp^{5}-S^{1}\times D^{4}),$$
and $\rp^{5}-S^{1}\times D^{4}$ is the disc bundle $D(2\eta)$ over $\rp^{3}$. Therefore $X^{5}(0)$ is actually the sphere bundle $S(2\eta \oplus \mathbb R)$. The cohomology groups are easily computed and we see that $\langle \alpha ^3 \,\cup\, \beta  , f_
*[X] \rangle =1$.

Now let $r \colon X^{5}(0) \to X^{5}(0)$ be the fiberwise antipodal map, we have a commutative diagramm
$$\xymatrix{
X^{5}(0) \ar[r]^{f} \ar[d]^{r} & \rp^{\infty} \times \cp^{\infty} \ar[d]^{(\mathrm{id}, \tau)} \\
X^{5}(0) \ar[r]^{f} & \rp^{\infty} \times \cp^{\infty}. \\}$$
Since $r$ is orientation reversing, we conclude that the action of $\tau$ on the $\z/4$ factor is multiplication by $-1$. It's also clear that the action of $\tau$ on the $\Omega _4^{\Pin^+}$ is trivial.
\end{proof}

In the general situation, the calculation is similar, and we have

\begin{Proposition}[type III]\label{prop:three}
The bordism group
$ \Omega _5^{\Spin}(\rp^{\infty}
\times (\cp^{\infty})^r;p_{1}^*(2\eta))$ is isomorphic to   $(\z/4)^r \oplus (\z/2)^{r(r-1)/2} \oplus \Omega _4^{\Pin^+}$. Furthermore,
\begin{enumerate}
\item the $\z/2$-factors are determined by the invariants $\langle \alpha \,\cup\,
\beta_i \,\cup\, \beta_j, f_
*[X] \rangle \in \z/2$, with $i, j=1, \cdots r$, and $i>j$,
\item a bordism class $[X,f]$ has component $\pm 1$ in the $i$-th $\z/4$-factor if and only if  $\langle \alpha ^3 \,\cup\, \beta _i , f_
*[X] \rangle =1 \in \z/2$, $i=1, \cdots r$,
\item there are relations $\langle \alpha \,\cup\, \beta_i^2 , f_
*[X] \rangle = \langle \alpha ^3 \,\cup\, \beta_i, f_*[X] \rangle $ for all $i$,
\item  the action $\tau_{i}$ on the bordism group is the multiplication by $-1$ on the $i$-th $\z/4$-factor and trivial on other factors.

\end{enumerate}
\end{Proposition}

\medskip
\noindent
\underline{Type I:} recall that the normal $2$-type is
$$p\colon  B=\rp^{\infty} \times
(\cp^{\infty})^r \times B\Spin \to BO,$$ where the map $p$ on the
first $\cp^{\infty}$ is the classifying map of the vector bundle
$\gamma$. Therefore the bordism group $\Omega_ 5{(B,p)}$ is the
twisted $\Spin$-bordism group
$$\Omega _5^{\Spin}(\rp^{\infty}
\times (\cp^{\infty})^r;p_{2}^*\gamma) = \widetilde{\Omega}_
7^{\Spin}(\Th(p_{2}^*\gamma)).$$

As before we apply the Thom isomorphism and the
$E^2$-terms in the Atiyah-Hirzebruch spectral sequence are
$E^2_{p,q}=\w H_p(\rp^{\infty} \times
(\cp^{\infty})^r;\Omega_q^{\Spin})$, where in the identification of the differentials $d _2$, we replace $Sq^{2}$ by $Sq^2+w_ 2(\gamma)$. The calculation is analogous to the type II case.

\begin{Proposition}[type I]\label{prop:four}
The bordism group $\Omega_{5}^{\Spin}(\rp^{\infty} \times (\cp^{\infty})^{r}; p_{2}^{*}\gamma)$ is isomorphic to $(\z/4)^{r-1} \oplus (\z/2)^{r(r-1)/2} \oplus \Omega_ 4^{\Pin^c}$.
Furthermore,
\begin{enumerate}
\item the $\z/2$-factors are determined by the invariants $\langle \alpha \,\cup\,
\beta_i \,\cup\, \beta_j, f_
*[X] \rangle \in \z/2$, with $i, j=1, \cdots r$, and $i>j$,
\item a bordism class $[X,f]$ has component $\pm 1$ in the $i$-th $\z/4$-factor if and only if $\langle \alpha ^3 \,\cup\, \beta _i , f_
*[X] \rangle =1  \in \z/2$, $i=2, \cdots ,  r$,
\item there are relations $\langle \alpha^5 + \alpha^3 \,\cup\, \beta_1, f_*[X] \rangle =0$ and $\langle \alpha \,\cup\, \beta_i^2 , f_
*[X] \rangle = \langle \alpha \,\cup\, \beta_1 \,\cup\, \beta_i, f_*[X] \rangle $ for all $i$,
\item the action $\tau_{i}$ ($ i \ge 2$) on the bordism group is the multiplication by $-1$ on the $i$-th $\z/4$-factor and trivial on other factors.

\end{enumerate}
\end{Proposition}

\section{Proofs of the Main Results}\label{sec:five}
In this section we prove Theorem \ref{thm:one} and Theorem \ref{thm:three}. From the point of view of Propositioni \ref{prop:one}, the key point to prove Theorem \ref{thm:one} is to show that for manifolds having the same invariants stated in the theorem, we can find  appropriate normal $2$-smoothings in $B$, such that they are bordant in $\Omega_{5}{(B,p)}$. (In some applications, this is done by understanding  the action of the group of fiber homotopy equivalences $Aut(B,p)$ on $\Omega_{n}{(B,p)}$. But in our situation, we find it more practical to find the smoothings directly.)
\begin{Lemma}\label{lemma:basis}
Let $M^5$ be a \abelian manifold with $\pi_1 (M) \cong \z/2$ and $H
_2(M) \cong \z^r$. Let $t \in H^1(M;\z/2)$ be the nonzero element,
and let $\{t^2, x_1, \cdots ,x_r \}$ be a basis of $H^2(M;\z/2)$. Then
$\{t^3, tx_1, \cdots, tx_r\}$ is a basis of $H^3(M;\z/2)$.
\end{Lemma}
\begin{proof}
Consider the Leray-Serre cohomology spectral sequence for the
fibration $\widetilde{M} \to M \to \rp^{\infty}$ with
$\z/2$-coefficients. Note that $\dim H^2(M;\z/2) =r+1$ and $\dim
H^2(\widetilde{M};\z/2) =r$. This implies that the differential $$d_
2\colon  E_ 2^{0,2}=H^2(\widetilde{M};\z/2) \to E
_3^{3,0}=H^3(\rp^{\infty};\z/2)$$ must be  trivial. Therefore, the
elements $t^3, tx_1, \cdots , tx_r$ all survive to form a basis of
$H^3(M;\z/2)$.
\end{proof}

\medskip
\begin{proof}[{The proof of Theorem  \ref{thm:one}}]
First of all, by Lemma \ref{lem:Pbordant}, we see that in the type
II and III cases, $[P] \in \Omega_4^{\Pin^\pm}/\{\pm 1\}$
 is an invariant for
$M$. Since we don't have a statement for $\Pin^c$, we will give an
alternative argument below for the type I case.

Let $f \colon  M^5 \to \rp^{\infty}$ be the classifying map of
$\pi_{1}$, $t=f^*\alpha \in H^{1}(M;\z/2)$. Consider the
nondegenerate symmetric bilinear form
$$\lambda \colon H^2(M;\z/2) \times H^2(M;\z/2)
\stackrel{\,\cup\,}{\rightarrow} H^4(M;\z/2) \stackrel{\,\cup\,
t}{\rightarrow} H^5(M;\z/2) \cong \z/2.$$

\smallskip
\noindent
\underline{Type II:} note that since $\langle t^{5}, [M] \rangle =
\langle \alpha^{5}, f_{*}[M] \rangle$ and
$\Omega_{5}^{\Spin}(\rp^{\infty})=0$, we have $\lambda(t^2,t^2)=0$.
From this and the relations in Proposition \ref{prop:two} we see
that $\lambda (x,x)=0$ for all $x$. Therefore we may extend $t^{2}$
to a symplectic basis of $\lambda$, $\{t^2, u_1, \cdots , u_r\}$.
Especially we have $\lambda(t^2, u_1)=1$, $\lambda(u_{1}, u_{1})=0$
and $\lambda(t^2, u_i)=\lambda(u_{1}, u_{i})=0$ for $i>1$. Now let
$u_i'=u_i+u_1$ for $i>1$, then $\lambda(t^2, u_i')=1$ for all $i$
and $\lambda(u_i', u_j')=\lambda(u_i, u_j)$. We may lift $\{u_1',
\cdots , u_r'\}$ to a basis of the free part of $H^2(M)$ and get a
map $M \to (\cp^{\infty})^{r}$.  Together with the canonical map $f
\colon M \to \rp^{\infty}$ and the classifying map of a
Spin-structure $M \to B\Spin$, we obtain a normal $2$-smoothing
$\bar{\nu} \colon M \to B=\rp^{\infty} \times (\cp^{\infty})^{r}
\times B\Spin$.

Now suppose $M'$ is another manifold, with a normal $2$-smoothing
$\bar{\nu}'$ constructed as above, then by Proposition
\ref{prop:two}, (composing $\bar{\nu}'$ with some $\tau_{i}$ to
interchange $\pm 1$ in the $\z/4$-factors if necessary) $[M,
\bar{\nu}] = [M', \bar{\nu}'] \in \Omega_{5}^{\Spin}(\rp^{\infty}
\times (\cp^{\infty})^{r})$. Proposition \ref{prop:one} implies that
they are diffeomorphic.

\smallskip

For the other two cases, the procedure of finding an appropriate map
to $(\cp^{\infty})^r$ is similar, thus we will omit the details.

\smallskip
\noindent
\underline{Type III:} first note that by the relation in Proposition
\ref{prop:three}, for all $x \in H^2(M;\z/2)$, $\lambda(t^2,
x)=\lambda(x,x)$. There are two different cases:
\begin{enumerate}
\item if $\lambda(t^2, t^2)=0$: then there exists a $u_1$ such that
$\lambda(t^2, u_1)=1$. On the orthogonal complement of
$\mathrm{span}(t^2, u_1)$, we have $\lambda(x,x)=0$, thus there
exists a symplectic basis $\{u_2, \cdots, u_r\}$. Then the argument
is the same as in the previous case.

\item if $\lambda(t^2, t^2)=1$: let $U$ be the orthogonal complement of
$\mathrm{span}(t^2)$, then $\lambda(x,x)=\lambda(t^2, x)=0$ for all $x \in U$.
There exists a symplectic basis of $U$, $\{u_1, \cdots, u_r\}$. Let $u_i'=u_i+t^2$,
then $\lambda(t^2,u_i')=1$ for all $i$ and $\lambda(u_i', u_j')=\lambda(u_i,u_j)+1$.
The remaining argument is the same as in the previous case.
\end{enumerate}

For $M$ and $M'$ having the same $\rk H_2=r$, like in the
type II case, we may use maps to $(\cp^{\infty})^r$ constructed
above to make the corresponding bordism classes have equal $(\z/4)^r
\oplus (\z/2)^{r(r-1)/2}$-component. Now if $M$ and $M'$ have
$[P]=[P'] \in \Omega_4^{\Pin^+}/\pm$, then by choosing an
appropriate Spin-structure on $TM \oplus f^*(2\eta)$, we may make
the $\Omega_4^{\Pin^+}$-component equal.
\smallskip

\smallskip
\noindent
\underline{Type I:} let $u_1=w_2(M)$. By the relation in Proposition
\ref{prop:four}, for all $x \in H^2(M;\z/2)$, $\lambda(u_1,
x)=\lambda(x, x)$. To find the map to $(\cp^{\infty})^r$, we have
four cases:

\begin{enumerate}
\item if $\lambda(t^2, t^2)=1$ and $\lambda(u_1, u_1)=0$: then $\lambda$ is nondegenerate
on $\mathrm{span}(t^2,u_1)$. Let $U$ be the orthogonal complement of
$\mathrm{span}(t^2,u_1)$, then and for all $x \in U$
 $\lambda(x,x)=0$. There
exists a symplectic basis $\{u_2, \cdots, u_r\}$.

\item if $\lambda(t^2,t^2)=0$ and $\lambda(u_1, u_1)=0$: then exists a $u_2$ such that
$\lambda(u_1, u_2)=1$ and $\lambda(t^2, u_2)=0$. $\lambda$ is
nondegenerate on $\mathrm{span}(u_1, u_2)$. On the orthogonal
complement we have $\lambda(x,x)=0$. Therefore there is a symplectic
basis $\{t^2, u_3, \cdots, u_r\}$.

\item if $\lambda(t^2, t^2)=1$ and $\lambda(u_1, u_1)=1$: let $U$
be the orthogonal complement of $\mathrm{span}(u_1)$, then there
exists a symplectic basis $\{u_2, u_3, \cdots, u_r\}$ for $U$ and we
may choose $u_2=t^2 + u_1$.

\item if $\lambda(t^2, t^2)=0$ and $\lambda(u_1, u_1)=1$: then on the orthogonal
complement of $\mathrm{span}(x_1)$ there is a symplectic basis
$\{t^2, u_2, \cdots, u_r\}$.
\end{enumerate}

Now we need to consider the $\Omega_4^{\Pin^c}$-component. Note that
since the manifolds in the list given in Theorem \ref{thm:three}
exhaust all possible values of $\mathrm{rank}H_2$ and $[P]$, an $M$
of type I must be diffeomorphic to some manifolds in the list. Now
we just need to show that the manifolds in the list are not
diffeomorphic to each other.

The $s$-component of $[P]\in \Omega_4^{\Pin^c}$ is determined by
$w_2(P)^2$, therefore varying $\Pin^c$-structures on $P^4$ will not
change the $s$-component. Thus we see that the two subfamilies
$$X^5(q) \scs(S^2\times \rp^3) \scs ((\sharp_k S^2\times S^2) \times
S^1)$$ and
$$X^5(q) \scs(\cp^2 \times S^1) \scs ((\sharp_k S^2\times S^2) \times
S^1)$$ don't have coincidence.

Let $Q^4$ be a characteristic submanifold of $X^5(q)$, then a
characteristic submanifold of $X^5(q) \scs(S^2\times \rp^3) \scs
((\sharp_k S^2\times S^2) \times S^1)$ can be taken as $P=Q \sharp
(S^2 \times \rp^2) \sharp (S^2 \times S^2)$. We have $[S^2 \times
\rp^2]=[S^2 \times S^2]=0 \in \Omega_4^{\Pin^c}$. So we see $[P]= q
\in \Omega_4^{\Pin^c}/\pm$ and different $q$'s give
non-diffeomorphic $X^5(q) \scs(S^2\times \rp^3) \scs ((\sharp_k
S^2\times S^2) \times S^1)$. Similar for the manifolds $X^5(q)
\scs(\cp^2 \times S^1) \scs ((\sharp_k S^2\times S^2) \times S^1)$,
since $[\cp^2]=(0,1) \in \Omega_4^{\Pin^c}$.

\end{proof}

The relations among the invariants are essentially seen in the
previous proof, but there is a more conceptual way to see this.

\begin{proof}[The proof of Theorem  \ref{thm:three}]
We will use the semi-characteristic class defined by R.~Lee in
\cite{lee1}. We work with $\mathbb Q$-coefficient, in this case, the
semi-characteristic class of an odd dimensional manifold with a free
$\z/2$-action is a homomorphism
$$\chi_ {1/2}\colon  \Omega_ 5(\z/2) \to  L^5(\mathbb Q[\z/2]) \cong
\z/2 ,$$ where $\Omega_ 5(\z/2)$ is the bordism group of closed
smooth oriented manifolds with an orientation-preserving free
$\z/2$-action, and $L^5(\mathbb Q[\z/2])$ is the symmetric $L$-group
of the rational group ring $\mathbb Q[\z/2])$. We refer to
\cite{lee1} and \cite{davis-milgram2} for details.

Let $M^5$ be an oriented smooth $5$-manifold with fundamental group
$\z/2$, then the semi-characteristic class
$\chi_ {1/2}(\widetilde{M};\mathbb Q) \in \z/2$ is defined. There is
a characteristic class formula \cite[Theorem C]{davis-milgram2}
$$\chi_ {1/2}(\widetilde{M};\mathbb Q)=\langle w_ 4(M)\,\cup\,
f^*(\alpha), [M]\rangle ,$$ where $f\colon  M \to \rp^{\infty}$ is
the classifying map of the covering and $\alpha \in
H^1(\rp^{\infty};\z/2)$ is the nonzero element. On the other hand,
$\chi_ {1/2}(\widetilde{M};\mathbb Q)$ is identified with (see
\cite[p.57]{davis-milgram2})
$$\begin{array}{cl}
\hat{\chi}_ {1/2}(\widetilde{M};\mathbb Q)  & : =\dim _{\mathbb Q}H _0(\widetilde{M};\mathbb Q) +\dim_ {\mathbb
Q}H _1(\widetilde{M};\mathbb Q)+\dim_ {\mathbb
Q}H_ 2(\widetilde{M};\mathbb Q) \pmod{2}\\
 & \equiv 1+r \pmod{2} .\\
\end{array}
$$

\smallskip
\noindent
\underline{Type II:} the Wu classes of $M$ are $v_ 1=0$ and $v_ 2=0$
since $w_ 1(M)=w_ 2(M)=0$. Therefore $w_ 4(M)=Sq^2v_ 2=0$. This means
$r$ is odd.

\smallskip
\noindent
\underline{Type III:} the Wu classes of $M$ are $v_ 1=0$ and $v_
2=w_ 2(M)=t^2$. Therefore $w_ 4(M)=Sq^2v_ 2=t^4$ and $\langle w_
4(M)\,\cup\, f^*(\alpha), [M]\rangle =\langle \alpha^5, \bar{\nu}_ *[M]
\rangle$. By the Atiyah-Hirzebruch spectral sequence, there is a
nonsplit exact sequence
$$0 \to \z/8 \to \Omega_ 5^{\Spin}(\rp^{\infty};2\eta) \to
H _5(\rp^{\infty}) \to 0.$$ Note that the bordism class $[M,
\bar{\nu}] \in \Omega_5^{\Spin}(\rp^{\infty}; 2\eta)$ corresponds to
the $\Pin^+$-bordism class of a characteristic submanifold, which
we denote by $q$. Therefore $\bar{\nu}_
*[M] \equiv q \pmod{2}$. This implies $r+q$ is odd.

\smallskip
 \noindent
 \underline{Type I:} the Wu classes of $M$ are $v_ 1=0$ and
$v _2=w_ 2(M)=\bar{\nu}^*w_ 2(\gamma)$. Therefore $w_ 4(M)=Sq^2v_
2=\bar{\nu}^*w_ 2(\gamma)^2$ and $\langle w_ 4(M)\,\cup\, f^*(\alpha),
[M]\rangle =\langle \alpha \,\cup\, \beta^2, \bar{\nu}_
*[M] \rangle$. Check on the generators of $\Omega_ 5^{\Spin}(\rp^{\infty} \times \cp^{\infty};\gamma)$, $\rp^5\scs(S^2\times \rp^3)$ with
$(q=1, s=0)$ and $\rp^5\scs (\cp^2 \times S^1)$ with $(q=1, s=1)$,
it is seen that $\langle \alpha \,\cup\, \beta^2, \bar{\nu}_
*[M] \rangle  \equiv q+s \pmod{2}$. This
implies the relation $q+s+r \equiv 1 \pmod{2}$.
\end{proof}

\section{Circle Bundles over $1$-connected $4$-manifolds}
\label{sec:six} As an application of the main results, in this section we study the classification of certain circle bundles over simply-connected $4$-manifolds.

Let $X^4$ be a simply-connected $4$-manifold,
smooth or topological,  and let $\xi$ be a complex line bundle over $X$, 
with first Chern class $c_ 1(\xi) \in H^2(X;\z)$. Choose a
Riemannian metric on $\xi$, and then the total space of the
corresponding circle bundle is a $5$-manifold $M$. The homotopy long
exact sequence of the fiber bundle shows that $\pi_ 1(M) \cong \z
/m$ if $\cxi$ is an $m$-multiple of a primitive element.

In
\cite{duan-liang1}, a classification of $M$ in terms of the topological
invariants of $X$ and $\cxi$ is obtained for $m=1$, using the
classification theorem of Smale and Barden. It is also known that
$H_2(M)$ is torsion-free of  $\rk H_ 2(X)-1$ and that $M$ is of fibered type.
 In this section, we will apply the classification results
to the $m=2$ case, to give classification of $M$ in terms of the
topological invariants of $X$ and $\cxi$. We will also identify $M$ in
the list of standard forms in Theorem  \ref{thm:four} and Theorem
\ref{thm:four.five}.

\subsection{Invariants of $M$}
In this subsection we collect the basic algebraic-topological
invariants of $M$.
\begin{Proposition}
Let $M^5$ be a circle bundle over a simply-connected $4$-manifold
$X$, with first Chern class $c_1(\xi)= 2\cdot \textrm{primitive}$,
then
\begin{enumerate}
\item
$\pi_ 1(M) \cong \mathbb Z/2$

\item
$H_ 2(M) \cong \mathbb Z^r$ where $r= \rk H_ 2(X)-1$.

\item
the $\pi_ 1(M)$--action on $H _2(\w{M})$ is trivial.

\item
the type of $M^5$ is given by

\medskip
\begin{center}
\begin{tabular}{|c|c|c|}
\hline
  type \textup{I} & type \textup{II} & type \textup{III} \\[0.3ex]
\hline
\vsp
$w_ 2(X) \ne 0$ & & \\[0.3ex]
 $w_ 2(X) \not \equiv c_ 1(\widetilde{\xi}) \pmod{2}$
&
$w_ 2(X)=0$ & $w_ 2(X) \equiv c_ 1(\widetilde{\xi}) \pmod{2}$\\[0.3ex]
\hline
\end{tabular}
\end{center}
\end{enumerate}
\end{Proposition}

\begin{proof}
First of all, the homotopy long exact sequence
$$\pi_ 1(S^1) \to \pi_ 1(M) \to \pi_ 1(X)$$ implies that $\pi_ 1(M)$ is
a cyclic group. The Gysin sequence
$$0 \to H_ 2(M) \to H_ 2(X) \xrightarrow{\cap\, c_ 1}
 H_ 0(X) \to H_ 1(M) \to 0$$
shows that $H_ 2(M)$ is torsion-free of rank equal to $\rk H_ 2(X) -1$
and $H _1(M) \cong \z/2$ since $c_ 1(\xi)= 2\cdot (\textrm{primitive})$.
Note that the universal cover $\widetilde{M}$ is a circle bundle
over $X$, denoted by $\widetilde{\xi}$, with first Chern class
$c_ 1(\widetilde{\xi})=\frac{1}{2}c_ 1(\xi)$. The $\pi _1(M)$-action on
$\widetilde{M}$ is the antipodal map on each fiber, and thus the
commutative diagram
$$\xymatrix{
H_ 2(\widetilde{M}) \ar[r]^{T_ *} \ar[dr]^{p_*}  & H_ 2(\widetilde{M}) \ar[d]^{p_*}\\
                                           & H_ 2(X)}
                                           $$
shows that the action on $H_ 2(\widetilde{M})$ is trivial. For the
Stiefel-Whitney class, if $X$ is smooth, we have $TM \oplus \mathbb
R =p^*(TX \oplus \xi)$ (where $p$ is the projection map), this
implies $w_ 2(M)=p^*w _2(X)$. In general, $X-pt$ admits a smooth
structure, then the same argument holds, see \cite[Lemma 3]{duan-liang1}.
\end{proof}

\subsection{Smoothings of $M$}
\begin{Proposition}
Let $\xi\colon  S^1 \hookrightarrow M^5 \to X$ be a nontrivial
circle bundle over a closed, simply-connected, topological
$4$-manifold. If $c_1(\xi)$ is an odd multiple of a primitive
element, then $M$ is smoothable; if $c_1(\xi)$ is an even multiple
of a primitive element, then $M$ admits a smooth structure if and
only if $KS(X)=0$.
\end{Proposition}

\begin{proof}
Let $M^5$ be a topological $5$-manifold, then by
\cite{kirby-siebenmann1}, the obstruction for smoothing $M$ lies in
$H^4(M;\pi_3(Top/O))=H^4(M;\pi_3(Top/PL))=H^4(M; \mathbb Z/2) \cong
H_1(M;\z/2)$. The latter group is trivial if $c_1(\xi)$ is an odd
multiple of a primitive element. On the other hand, we have
$TM\oplus \mathbb R =\pi^*(TX \oplus \xi)$, where $\pi$ is the
projection map. Therefore the obstruction for smoothing $M$ is
$\pi^*KS(X)$. It is seen from the Gysin sequence that $\pi^*\colon
H^4(X;\z/2) \to H^4(M;\z/2)$ is injective if $c_1(\xi)$ is an even
multiple of a primitive element. Therefore $M$ admits a smooth
structure if and only if $KS(X)=0$.
\end{proof}

Now we give a geometric description of the characteristic
submanifold of a circle bundle over simply-connected $X^4$.

\begin{Lemma}\label{three}
Let $\xi\colon  S^1 \hookrightarrow M^5 \to X$ be a circle bundle,
$\pi_1(M) \cong \z /2$. Let $F\subset X$ be an embedded surface dual
to $c_1(\widetilde{\xi})$, $N(F)$ be a tubular neighborhood of $F$
in $X$, $S^1 \hookrightarrow B \to F$ be the restriction of $\xi$ on
$F$. Then there is a double cover map $\partial N(F) \to B$ and the
characteristic submanifold of $M$ is $P^4=(X-\mathring{N}(F))\,\cup_
{\partial}\, B$.
\end{Lemma}

In other words, the characteristic submanifold $P$ is obtained by
removing a tubular neighborhood of an embedded surface dual to
$c_1(\widetilde{\xi})$ and then identifying antipodal points on on
each fiber.

\begin{proof}
Since $c_1(\xi)=2 \cdot$(primitive), the circle bundle is the
pull-back of the circle bundle over $\cp ^2$ with first Chern class
$= 2\cdot$(primitive):
$$\xymatrix@R-4pt{
S^1 \ar[r]^= \ar[d] & S^1 \ar[d] & \\
M^5 \ar[r]^f \ar[d] & \rp ^5 \ar[d] & \supset\ \rp^4=\rp^3 \cup D^4 \\
X \ar[r]^g & \cp ^2 & = \cp^1 \cup D^4 \\}$$ Now
$P=f^{-1}(\rp^4)=f^{-1}(D^4\cup_{S^3}\rp ^3)$. Let $F=g^{-1}(\cp
^1)$ be the transvere preimage of $\cp ^1$, then the normal bundle
$\nu$ of $F$ in $X$ is the pullback of the Hopf bundle, and the
restriction of $\xi$ on $F$ is $\nu \otimes \nu$, therefore there is
a double cover $\partial N(F) \to B$. It is easy to see that
$P^4=(X-\mathring{N}(F)\,\cup _{\partial}\, B$.
\end{proof}

\begin{Lemma}\label{KS}
Let $P$ be as above. Then $KS(P)=KS(X)$.
\end{Lemma}

\begin{proof}
We identify $N(F)$ with the normal $2$-disk bundle, let $V$ be the
associated $\rp^2$-bundle obtained by identifying antipodal points
on $\partial N(F)$. Then by the construction,
$$P=X\cup_{N(F)\times \{0\}} N(F) \times I \cup_{N(F) \times \{1\}}V.$$
Therefore $P$ is bordant to $X \,\sqcup \,V$. It was shown by Hsu
\cite{hsu1} and Lashof-Taylor \cite{lashof-taylor1} that the Kirby-Siebenmann
invariant is a bordism invariant, thus $KS(P)=KS(X)+KS(V)=KS(X)$
since $V$ is smooth.
\end{proof}

\subsection{Classification}

Now we can give a classification of circle bundles over
$1$-connected $4$-manifolds, and identify them with the standard forms
in Theorem  \ref{thm:four} and Theorem \ref{thm:four.five}, in terms
of the topology of $X$ and $\xi$.

\smallskip

For the type II manifolds it is an immediate consequence of Theorem
\ref{thm:one} and Theorem \ref{thm:two}.

\begin{Theorem}[type II]\label{thm:typetwo}
Let $X$ be a closed, simply-connected, topological spin
$4$-manifold, $\xi\colon  S^1 \hookrightarrow M^5 \to X$ be a circle
bundle over $X$ with $c_1(\xi)=2\cdot$(primitive). Then we have
\begin{enumerate}\addtolength{\itemsep}{0.3\baselineskip}
\item if $KS(X)=0$, then $M$ is smoothable and $M$ is diffeomorphic
to $$(S^2 \times \mathbb R \mathrm P^3) \scs ((\sharp_k\,  S^2
\times S^2)\times S^1);$$

\item if $KS(X)=1$, then $M$ is non-smoothable and $M$ is
homeomorphic to $$*(S^2 \times \mathbb R \mathrm P^3) \scs
((\sharp_k\,  S^2 \times S^2)\times S^1).$$
\end{enumerate}
Where $k=\rk H_2(X)/2-1$.
\end{Theorem}

\begin{remark} Note that for a spin $4$-manifold $X$,
 $\rk H_2(X)$ is even, and thus $k$ is an integer.
\end{remark}

For smooth manifolds of type III, we do not know a good invariant
detecting the bordism group $\Omega_4^{\Pin^+}$. Therefore we could
only determine the diffeomorphism type up to an ambiguity of order
2. This is based on the following exact sequence (see \cite[\S
5]{kirby-taylor2})
$$ 0 \to \z /2 \to \Omega_4^{\Pin^+} \stackrel{\cap\, w_1^2}{\longrightarrow} \Omega_2^{\Pin^-} \to 0,$$
where $\cap\, w_1^2$ is the operation of taking a submanifold dual to
$w_1^2$. The generators of $\Omega_2^{\Pin^-}$ is $\pm \rp^2$ and
$\cap\, w_1^2$ maps $\pm \rp^4$ to $\pm \rp^2$. The image of $[P]$ in
$\Omega_2^{\Pin^-}$ can be determined from the data of the circle
bundle.

In the topological case, we have an epimorphism (see \cite[\S
9]{kirby-taylor2})
$$\Omega_4^{\TopPin^+} \to \Omega_2^{\TopPin^-} \cong \z/8 ,$$
which is an isomorphism on the subgroup generated by $\rp^4$. By
Lemma \ref{KS}, we have $KS(P)=KS(X)$. Therefore by Theorem
\ref{thm:two}, we have a complete topological classification.

\begin{Theorem}[type III]\label{thm:typethree}
Let $X$ be a closed, simply-connected  topological $4$-manifold, and let
$\xi\colon  S^1 \hookrightarrow M^5 \to X$ be a circle bundle over
$X$ with $c_1(\xi)=2\cdot$(primitive) and $w_2(X) \equiv
c_1(\widetilde{\xi}) \pmod{2}$. Then we have
\begin{enumerate}\addtolength{\itemsep}{0.3\baselineskip}
\item if $X$ is smooth, then the diffeomorphism type of $M$ (with the induced smooth structure) is determined up to
an ambiguity of order $2$ by $\rk H_2(X)$ and $ \langle
c_1(\widetilde{\xi})^2, [X]\rangle \in (\z/8)/\pm =\{0,1,2,3,4\}$.

\item $M$ is homeomorphic to $X^5(p,q) \scs ((\sharp_k\,  S^2 \times S^2)\times S^1)$, where $q =
\langle c_1(\widetilde{\xi})^2, [X]\rangle \in (\z/8)/\pm
=\{0,1,2,3,4\}$, $k=(\rk H_2(X)-(3+(-1)^q)/2)/2$, $p=KS(X)$.
\end{enumerate}
\end{Theorem}

\begin{proof}
We only need to prove (1), since the proof of (2) is similar. We see from
the proof of Lemma \ref{three} that $P=f^{-1}(\rp^4)$, where
$f\colon P \to \rp^4$ induces an isomorphism on $\pi_ 1$. If the mod
$2$ degree of $f$ is $1$, then the submanifold dual to $w_ 1(P)$ is
$f^{-1}(\rp^3)$, and the submanifold $V$ dual to $w_ 1(P)^2$ is
$f^{-1}(\rp^2)$. Now we have the following commutative diagram
$$\xymatrix{
S^1 \ar[r]^= \ar[d] & S^1 \ar[d] & \\
\partial N(F) \ar[r]^f \ar[d] & \rp ^3 \ar[d] & \ \supset \rp^2=D^2 \cup S^1 \\
F \ar[r]^g & \cp ^1 & = D^2 \cup pt \\}$$
Let $d=\deg g=\langle
c_ 1(\widetilde{\xi})^2, [X]\rangle$ and $D=g^{-1}(pt)=\{p_ 1,\cdots,
p_ d\}$, it is seen that $V=f^{-1}(\rp^2)=(F-D)\cup_ {\partial}\,d\cdot
S^1$ (where the glueing map is of degree $2$) and
$[V]=d\cdot[\rp^2]\in \Omega_ 2^{\Pin^-}$. If the mod $2$ degree of
$f$ is zero, then we consider the circle bundle over $X \Sharp
\cp^2$ with first Chern class $(c_ 1(\xi),2)$. The corresponding map
has nonzero mod $2$ degree, the image of the corresponding
characteristic submanifold in $\Omega_ 2^{\Pin^-}$ equals to that of
the original one plus $1$. Finally $\langle
(c _1(\widetilde{\xi}),1)^2, [X \Sharp \cp^2]\rangle = \langle
c_ 1(\widetilde{\xi})^2, [X]\rangle+1$. This proves the theorem.
\end{proof}

For the manifolds of type I, we have
\begin{Theorem}[type I]\label{thm:typeone}
Let $X$ be a closed, simply-connected non-spin topological
$4$-manifold, and let $\xi\colon  S^1 \hookrightarrow M^5 \to X$ be a circle
bundle over $X$ with $c_ 1(\xi)=2\cdot$(primitive) and $w_ 2(X) \not
\equiv c _1(\widetilde{\xi}) \pmod{2}$. We have

\begin{enumerate}\addtolength{\itemsep}{0.3\baselineskip}
\item if $KS(X)=0$, then $M$ is smoothable and

\begin{itemize}\addtolength{\itemsep}{0.3\baselineskip}
\item if $\langle w_ 2(X)^2, [X]\rangle \equiv \langle c_
1(\widetilde{\xi})^2,[X]\rangle \pmod{2}$, then $M$ is diffeomorphic
to
$$\XX{q} \scs (S^2 \times \rp ^3) \scs ((\sharp_k\,  S^2 \times
S^2)\times S^1),$$ where $q = \langle c_ 1(\widetilde{\xi})^2,
[X]\rangle \in (\z/8)/\pm = \{0,1,2,3,4\}$ and
$$k=\frac{1}{2}(\rk H_ 2(X)-\frac{1}{2}(7+(-1)^q));$$

\item if $\langle w_ 2(X)^2, [X]\rangle  \not \equiv \langle c_
1(\widetilde{\xi})^2,[X]\rangle \pmod{2}$, then $M$ is diffeomorphic
to $$\XX{q} \scs (\cp^2 \times S^1) \scs ((\sharp _k\, S^2 \times
S^2)\times S^1),$$ where $q = \langle c_ 1(\widetilde{\xi})^2,
[X]\rangle \in (\z/8)/\pm = \{0,1,2,3,4\} $ and $$k=\frac{1}{2}(\rk
H_ 2(X)-\frac{1}{2}(5+(-1)^q)).$$
\end{itemize}

\item if $KS(X)=1$, then $M$ is non-smoothable and
\begin{itemize}\addtolength{\itemsep}{0.3\baselineskip}
\item if $\langle w_ 2(X)^2, [X]\rangle \equiv \langle c_
1(\widetilde{\xi})^2,[X]\rangle \pmod{2}$, then $M$ is homeomorphic
to
$$X^5(1,q) \scs (S^2 \times \rp ^3) \scs ((\sharp_k\,  S^2 \times
S^2)\times S^1),$$ where $q = \langle c_ 1(\widetilde{\xi})^2,
[X]\rangle \in (\z/8)/\pm = \{0,1,2,3,4\} $ and $$k=\frac{1}{2}(\rk
H_ 2(X)-\frac{1}{2}(7+(-1)^q));$$

\item if $\langle w_ 2(X)^2, [X]\rangle  \not \equiv \langle c_
1(\widetilde{\xi})^2,[X]\rangle \pmod{2}$, then $M$ is homeomorphic
to
$$X^5(1,q) \scs (\cp^2 \times S^1) \scs ((\sharp _k\, S^2 \times
S^2)\times S^1),$$ where $q = \langle c_ 1(\widetilde{\xi})^2,
[X]\rangle \in (\z/8)/\pm = \{0,1,2,3,4\} $ and $$k=\frac{1}{2}(\rk
H_ 2(X)-\frac{1}{2}(5+(-1)^q)).$$
\end{itemize}
\end{enumerate}
\end{Theorem}

\begin{remark} Note that for a $4$-manifold $X$, $\langle
w_ 2(X)^2, [X]\rangle \equiv \rk H_ 2(X) \pmod{2}$. This
ensures that $k$ is an integer.
\end{remark}

\begin{proof}
We only need to prove (1); the proof of (2) is similar. Recall that
we have $\Omega_ 4^{\Pin^c} \cong \z/8 \oplus \z/2$, with generators
$\rp ^4$ and $\cp^2$.  Thus the $q$-component is determined as in
the type III case. The $s$-component of $P$ is determined by the
bordism number $\langle w_ 2(P)^2, [P]\rangle \in \z/2$. (Here we
use the notations given before Theorem \ref{thm:three}.) Since
$KS(X)=0$, there exists an integer $m$ such that $X_0=X \Sharp m(S^2
\times S^2)$ is smooth. Note that if we do the same construction on
$X_0$ we get $P_0=P \Sharp m(S^2 \times S^2)$, and $\langle w_
2(P_0)^2, [P_0]\rangle=\langle w_ 2(P)^2, [P]\rangle$. Therefore,  to
compute the $s$-component, we may assume that $X$ is smooth. Recall
that $P=(X-\mathring{N}(F))\,\cup_ {\partial}\, B$, it is seen that the
bordism class of $P$ is determined by the bordism class of the pair
$(X,F)$, which can be viewed as a singular manifold $(X,f) \in
\Omega_ 4(BU(1)) \cong \Omega_ 4 \oplus H_ 4(BU(1))$. We have two
homomorphisms
$$\Omega_ 4(BU(1)) \to \z/2, \ \ \ [X,F] \mapsto \langle w_ 2(P)^2,
[P]\rangle$$ and
$$\Omega_ 4(BU(1)) \to \z/2, \ \ \ [X,c_ 1(\widetilde{\xi})] \mapsto \langle w_ 2(X)^2+c_ 1(\widetilde{\xi})^2,
[X]\rangle .$$ By a check on the generators $(\cp^2 \Sharp (S^2 \times
S^2), c_ 1(\widetilde{\xi})=(1,0,1))$ and $(\cp^2 \Sharp (S^2 \times
S^2), c_ 1(\widetilde{\xi})=(0,0,1))$, we see that $s=\langle w_
2(P)^2, [P]\rangle = \langle w _2(X)^2+c_ 1(\widetilde{\xi})^2,
[X]\rangle \pmod{2}$. The two cases correspond to the values $s=0$
and $s=1$. For the proof of (2), the only change is that
$\Omega_4^{\Top} \cong \z \oplus \z/2$ with generators $\cp^2$ and
$*\cp^2 \Sharp \overline{\cp^2}$ \cite{hsu1}.
\end{proof}

\providecommand{\bysame}{\leavevmode\hbox to3em{\hrulefill}\thinspace}
\providecommand{\MR}{\relax\ifhmode\unskip\space\fi MR }
\providecommand{\MRhref}[2]{%
  \href{http://www.ams.org/mathscinet-getitem?mr=#1}{#2}
}
\providecommand{\href}[2]{#2}

\end{document}